\documentclass[11pt]{amsart}
\usepackage{amsmath,amsfonts,amsbsy,amsgen,amscd,mathrsfs,amssymb,amsthm}
\usepackage{enumerate,mathtools}
\usepackage{cases}

\usepackage{esint}

\usepackage{clipboard}
\newclipboard{myclipboard}

\usepackage[T1]{fontenc} 

\usepackage{fullpage}
\usepackage{url}

\usepackage{mathrsfs}

\def\XXint#1#2#3{{\setbox0=\hbox{$#1{#2#3}{\int}$}
\vcenter{\hbox{$#2#3$}}\kern-.5\wd0}}

\usepackage[colorlinks=true,linkcolor=blue]{hyperref} 
\makeatletter
\renewcommand*{\eqref}[1]{%
  \hyperref[{#1}]{\textup{\tagform@{\ref*{#1}}}}%
}
\makeatother

\numberwithin{equation}{section}

\newtheorem{theorem}{Theorem}[section]
\newtheorem{theorem_informal}[theorem]{Theorem (informal)}
\newtheorem*{theorem*}{Theorem}
\newtheorem{proposition}[theorem]{Proposition}

\newtheorem*{corollary*}{Corollary}
\newtheorem{lemma}[theorem]{Lemma}
\theoremstyle{definition}
\newtheorem{definition}[theorem]{Definition}

\newtheorem{remark}[theorem]{Remark}

\newtheorem{assumption}[theorem]{Assumption}
\newtheorem*{claim*}{Claim}

\newcommand{\R}{\mathbb R}
\newcommand{\C}{\mathbb C}
\newcommand{\Id}{\operatorname{Id}}

\newcommand{\Tr}{\operatorname{Tr}}
\newcommand{\Cof}{\operatorname{Cof}}
\newcommand{\Cov}{\operatorname{Cov}}
\newcommand{\Div}{\operatorname{div}}
\newcommand{\Vol}{\operatorname{Vol}}
\newcommand{\V}{V}
\newcommand{\W}{W}
\newcommand{\prob}{\rho}
\newcommand{\source}{\mu}
\newcommand{\target}{\nu}
\newcommand{\lsh}{\alpha}
\newcommand{\lc}{\kappa}
\newcommand{\sta}{g}
\newcommand{\fin}{h}
\newcommand*\diff{\mathop{}\!\mathrm{d}}
\newcommand{\dd}{n}
\newcommand{\Pheat}{\mathrm{P}}
\newcommand{\Hheat}{\mathrm{H}}
\newcommand{\conv}{c}
\newcommand{\density}{f}
\newcommand{\Gaussian}{\mathrm \gamma}

\newcommand{\KM}{\mathrm F}
\newcommand{\OT}{\Phi}
\newcommand{\map}{\mathrm T}
\newcommand{\cotrans}{\mathcal L}
\newcommand{\wave}{\psi}
\newcommand{\waveg}{\psi_{\textnormal{Gaussian}}}

\newcommand{\cxfun}{\varphi}
\newcommand{\avgint}{\fint}
\newcommand{\findiff}{\delta\OT}
\newcommand{\Ent}{\mathrm H}

\title{Optimal transport maps, majorization, and log-subharmonic measures}

\author{Guido De Philippis}
\address{Courant Institute of Mathematical Sciences, New York University. 251 Mercer Street, New York, NY 10012-1185}
\email{guido@cims.nyu.edu}

\author{Yair Shenfeld}
\address{Division of Applied Mathematics, Brown University, Providence, RI, USA}
\email{Yair\_Shenfeld@Brown.edu}

\begin{document}
\maketitle

\begin{abstract}
Caffarelli's contraction theorem bounds the derivative of the optimal transport map between a  log-convex measure and a strongly log-concave measure. We show that an analogous phenomenon holds on the level of the trace: The  trace of the derivative of the  optimal transport map between a log-subharmonic measure and a strongly log-concave measure is bounded.  We show that this trace bound has a number of consequences pertaining to volume-contracting transport maps, majorization and its monotonicity along Wasserstein geodesics, growth estimates of log-subharmonic functions, the Wehrl conjecture for Glauber states, and two-dimensional Coulomb gases. We also discuss volume-contraction properties for the Kim-Milman transport map.
\end{abstract}

\section{Introduction}

\subsection{The divergence of the Brenier map}
The study of regularity properties of transport maps between probability measures is an important topic in analysis, probability, and geometry. The most classical result in the field is due to Caffarelli  \cite{MR1800860} who established Lipschitz regularity of the Brenier map of optimal transport  under log-convexity/concavity assumptions on the source and target measures. In this work we consider a low regularity notion of convexity:
\begin{definition}
\label{def:log_concave}
A probability measure $\prob$ on $\R^{\dd}$, with density $\diff \prob=e^{-U}\diff x$ with respect to the Lebesgue measure, is \emph{$\conv$-log-convex} (res. \emph{$\conv$-log-concave}), for $\conv\in \R$, if the distributional derivative $\nabla^2U$ satisfies $\nabla^2U  \preceq  \conv \Id_{\dd}$ (res. $\nabla^2U \succeq  \conv \Id_{\dd}$), where $\preceq$ (res. $\succeq$) are in the sense of positive semidefinite order. 
\end{definition}
It follows from Caffarelli's work that if the source measure is $\lsh$-log-convex, and the target measure is $\lc$-log-concave, then the Lipschitz constant of  the Brenier map between the source and target can be bounded in terms of $\lsh$ and $\lc$. 
\label{subsec:div_OT_intro}
\begin{theorem}[Caffarelli \cite{MR1800860}, Kolesnikov \cite{kolesnikov2011mass}]
\label{thm:Caffarelli_intro}
Let $\diff\source=e^{-\V}\diff x$ and $\diff\target=e^{-\W}\diff x$ be probability measures on $\R^{\dd}$, with $\source$ supported on all of $\R^{\dd}$, such that there exist $\lsh>0,\lc>0$ with 
\[
\nabla^2 \V \preceq \lsh \Id_{\dd}\quad\text{and}\quad \nabla^2\W \succeq \lc \Id_{\dd}.
\]
Let $\nabla\OT:\R^{\dd}\to\R^{\dd}$ be the Brenier map transporting $\source$ to $\target$. Then,
\begin{equation}
\label{eq:Caffarelli_intro}
\|\nabla^2\OT\|_{L^{\infty}(\diff x)}\le \sqrt{ \frac{\lsh}{\lc}},
\end{equation}
where $\|\cdot\|_{L^{\infty}(\diff x)}$ is the $L^{\infty}$ operator norm.
\end{theorem}
The bound \eqref{eq:Caffarelli_intro} is sharp as can be seen by taking Gaussians for the source and target, $\source=\mathcal N(0,\lsh^{-1}\Id_{\dd})$ and $\target=\mathcal N(0,\lc^{-1}\Id_{\dd})$. One special feature of \eqref{eq:Caffarelli_intro} is that the right-hand side is dimension-independent, which leads to numerous dimension-independent functional inequalities \cite{cordero2002some,cordero2004b,harge2004convex,klartag2007marginals}. For example, suppose $\source$ satisfies logarithmic Sobolev inequality with constant $c_{\source}$, that is, for every test function $f$,
\begin{equation}
\label{eq:lsi_source}
\int (\log f^2)f^2\diff \source-\int f^2\diff \source\,\log\int f^2\diff \source\le c_{\source}\int |\nabla f|^2\diff \source.
\end{equation}
Then, the bound \eqref{eq:Caffarelli_intro}  can be used to show that $\target$ satisfies a logarithmic Sobolev inequality with constant $\frac{\lsh}{\lc}c_{\source}$:
\begin{align}
\begin{split}
&\int (\log f^2)f^2\diff \target-\int f^2\diff \target\,\log\int f^2\diff \target\label{eq:Lip-func_ineq}\overset{(\nabla \OT)_{\sharp}\source=\target}{=}\int (\log (f\circ \nabla\OT)^2)(f\circ \nabla \OT)^2\diff \source\\
&-\int (f\circ \nabla\OT)^2\diff \source\,\log\int (f\circ\nabla \OT)^2\diff \source\overset{\eqref{eq:lsi_source}}{\le} c_{\source}\int |\nabla (f\circ \nabla\OT)|^2\diff \source\le c_{\source}\int \|\nabla^2\OT\|_{L^{\infty}}^2|(\nabla f)\circ \nabla\OT)|^2\diff \source\\
&\overset{\eqref{eq:Caffarelli_intro}}{\le}\frac{\lsh}{\lc}c_{\source}\int |(\nabla f)\circ \nabla\OT|^2\diff \source=\frac{\lsh}{\lc}c_{\source}\int |\nabla f|^2\diff \target.
\end{split}
\end{align}
 While Theorem \ref{thm:Caffarelli_intro} is highly useful there are numerous classes of probability measures which fall outside its scope as they do not satisfy the required convexity. In this work  we will focus on a particular such class of probability measures.
 \begin{definition}
\label{def:log_subharmonic}
A probability measure $\prob$ on $\R^{\dd}$, with density $\diff \prob=e^{-U}\diff x$ with respect to the Lebesgue measure where $U\in L_{\text{loc}}^1(\diff x)$, is \emph{$\conv$-log-subharmonic} (res. \emph{$\conv$-log-superharmonic}), for $\conv\in \R$, if the distributional derivative $\Delta U$ satisfies $\Delta U  \le\conv $ (res. $\Delta U \ge  \conv$). 
\end{definition}
Log-subharmonic measures arise naturally in complex-analytic settings since the modulus of a holomorphic function on $\mathbb C^{\dd}$ is log-subharmonic (cf. Theorem \ref{thm:Fock}). In particular,  we will see below (cf. Section \ref{subsec:Wehrl_intro}) that $\conv$-log-subharmonic measures are the right objects to investigate in the context of the generalized Wehrl conjecture. In addition, probability measures in  two-dimensional one-component plasma models are log-subharmonic  (cf. Section \ref{subsec:plasma_intro}). We refer the reader to \cite{MR2134401, MR2578455, MR3415657,MR4632741} for various  functional inequalities enjoyed by $\conv$-log-subharmonic measures.

Our first main result establishes an analogue of Theorem \ref{thm:Caffarelli_intro} on the level of the trace under log-subharmonicity assumptions. 
\begin{theorem}
\label{thm:main_OT_Lp_intro}
Let $\diff\source=e^{-\V}\diff x$ and $\diff\target=e^{-\W}\diff x$ be probability measures on $\R^{\dd}$, with $\source$ supported on all of $\R^{\dd}$,  such that there exist $\lsh>0,\lc>0$ with
\[
\Delta \V \le \lsh\dd \quad\text{and}\quad \nabla^2\W \succeq \lc \Id_{\dd}.
\]
Let $\nabla\OT:\R^{\dd}\to\R^{\dd}$ be the Brenier map transporting $\source$ to $\target$. Then, the  Laplacian $\Delta\OT$ satisfies
\begin{align}
\label{eq:infty_bound_main_intro}
\|\Delta\OT\|_{L^{\infty}(\diff x)}\le \dd\sqrt{ \frac{\lsh}{\lc}}.
\end{align}
\end{theorem} 
The bound \eqref{eq:infty_bound_main_intro} is sharp as can be seen by taking 
\[
\source=\mathcal N(0,\sigma^2\Id_{\dd}),\quad \target=\mathcal N(0,\Id_{\dd}); \quad  \lc=1, \quad \lsh=\frac{1}{\sigma^2},
\]
since in this case $\nabla \OT(x)=\frac{x}{\sigma}$ so 
\[
\Delta\OT(x)=\frac{\dd}{\sigma}=\dd\sqrt{\frac{\lsh}{\lc}}\quad \text{for all }x.
\]
In particular the right-hand side of  \eqref{eq:infty_bound_main_intro} must be dimension-dependent. Let us also remark that following the first version of this work,  Gozlan and  Sylvestre  showed \cite[\S 6]{gozlan2025global}  that our assumption on $\W$ can be weakened.

\begin{remark}[Lipschitz bounds]
\label{rem:Lipschiz}
Since the Brenier map $\nabla\OT$ between $\source$ and $\target$ is a gradient of a convex function, our bound \eqref{eq:infty_bound_main_intro} implies that under the assumptions of Theorem \ref{thm:main_OT_Lp_intro} we get the Lipschitz bound\footnote{A previous version of this paper erroneously claimed that this bound was sharp when $\source$ and $\target$ are Gaussians. However,  there was a mistake in the calculation and the optimal form of a bound
$\|\nabla^2\OT\|_{L^{\infty}(\diff x)}\le b(\Delta \V)$ for some function $b$, when the target $\target$ is standard Gaussian, seems to be unknown.}
\begin{equation}
\label{eq:Lipschitz}
\|\nabla^2\OT\|_{L^{\infty}(\diff x)}\le\dd \sqrt{ \frac{\lsh}{\lc}}.
\end{equation}
% The bound \eqref{eq:Lipschitz} is sharp as can be seen by taking 
%\[
%\source=\mathcal N(0,\Sigma_{\epsilon}),\quad \target=\mathcal N(0,\Id_{\dd}); \quad  \lc=1, \quad \lsh_{\epsilon}=1+\frac{\dd-1}{\dd}\epsilon,
%\]
%where $\Sigma_{\epsilon}$ is a diagonal matrix with one entry equal to $1/\dd$ and the rest of the entries equal to $1/\epsilon$ for  some $\epsilon\in (0,1]$. In this case $\nabla\OT(x)=\Sigma_{\epsilon}^{-1}x$ so that
%\[
%\|\nabla^2\OT\|_{L^{\infty}(\diff x)}\le \dd=\lim_{\epsilon\to 0}\dd \sqrt{ \frac{\lsh_{\epsilon}}{\lc}}.
%\]
%The bound \eqref{eq:Lipschitz} can be seen as the dimensional price that one has to pay if the log-convexity assumption in Theorem \ref{thm:Caffarelli_intro} is relaxed to log-subharmonicity assumption. 
\end{remark}

While Theorem \ref{thm:main_OT_Lp_intro} is the trace analogue of Theorem \ref{thm:Caffarelli_intro}, its applications to functional inequalities is different in nature. The argument in \eqref{eq:Lip-func_ineq} can no longer be used if we only have the bound  \eqref{eq:infty_bound_main_intro} to transfer functional inequalities (though one can always of course use the dimension-dependent bound \eqref{eq:Lipschitz}). Instead, the natural playground for trace bounds such as \eqref{eq:infty_bound_main_intro} is majorization.

\subsection{Volume contracting maps  and majorization}
\label{subsec:vol_contract_majorization}
The starting point of this section is the result of Melbourne and Roberto \cite{MR2023} on the relation between volume-contracting transport maps and majorization. To introduce their ideas we recall the basic definitions from the theory of majorization between probability measures \cite{marshall2011inequalities}. 

\begin{definition}
\label{def:majorization_intro}
A probability measure $\fin\diff x$ on $\R^{\dd}$ \emph{majorizes} a probability measure $\sta \diff x$ on $\R^{\dd}$ if
\[
\int_{\R^{\dd}} \cxfun(\sta)\diff x\le \int_{\R^{\dd}} \cxfun(\fin)\diff x\quad\text{for every convex function }\mathrm \cxfun:\R_{\ge 0} \to \R. 
\]  
\end{definition}

For example, we may take $\cxfun(x)=x\log x$ to get that the differential entropy of $\fin\diff x$ must be bigger than the differential entropy of $\sta\diff x$. The observation of Melbourne and Roberto is that the type of regularity of transport maps  relevant to majoirzation is volume-contraction.

\begin{definition}
A differentiable transport map $\map:\R^{\dd}\to \R^{\dd}$ between $\sta\diff x$ to $\fin\diff x$ is a \emph{$c$-volume-contraction} map if there exists a  constant  $c>0$ such that
\[
\|\det\nabla \map\|_{L^{\infty}(\sta\diff x)}\le c.
\]
\end{definition}

\begin{theorem}[Melbourne-Roberto \cite{MR2023}]
\label{thm:Melbourne-Roberto_intro}
Suppose there exists a transport map $\map:\R^{\dd}\to \R^{\dd}$ between probability measures $\sta\diff x$ and $\fin \diff x$ which is  a 1-volume-contraction map. Then,  $\fin\diff x$  majorizes $\sta \diff x$.
\end{theorem}
The connection between volume-contracting transport maps and Theorem \ref{thm:main_OT_Lp_intro} is a simple consequence
of the arithmetic-geometric-mean inequality.

\begin{theorem}[Volume contraction of the Brenier map]
\label{thm:volume-contraction_OT_intro}
Let $\diff\source=e^{-\V}\diff x$ and $\diff\target=e^{-\W}\diff x$ be probability measures on $\R^{\dd}$, with $\source$ supported on all of $\R^{\dd}$,  such that there exist $\lsh>0,\lc>0$ with
\[
\Delta \V \le \lsh\dd \quad\text{and}\quad \nabla^2\W \succeq \lc \Id_{\dd}.
\]
Then, the Brenier map $\nabla\OT$ transporting $\source$ to $\target$ is a $\left(\frac{\lsh}{\lc}\right)^{\frac{\dd}{2}}$-volume-contraction map,
\begin{align}
\label{eq:det_bound_main_intro}
 \|\det\nabla^2\OT\|_{L^{\infty}(\diff x)}\le \left(\frac{\lsh}{\lc}\right)^{\frac{\dd}{2}}.
 \end{align}
\end{theorem}

Combining Theorem \ref{thm:Melbourne-Roberto_intro} and Theorem \ref{thm:volume-contraction_OT_intro} we obtain majorization under log-subharmonicity.

\begin{theorem}[Majorization]
\label{thm:majorization_intro}
Let $\diff\source=e^{-\V}\diff x$ and $\diff\target=e^{-\W}\diff x$ be probability measures on $\R^{\dd}$, with $\source$ supported on all of $\R^{\dd}$,  such that there exist $\lsh>0,\lc>0$ with
\[
\Delta \V \le \lsh\dd \quad\text{and}\quad \nabla^2\W \succeq \lc \Id_{\dd}.
\]
If $\frac{\lsh}{\lc}\le 1$, then $\target$ majorizes $\source$.
\end{theorem}

In \cite{MR2023} the authors combined  Theorem \ref{thm:Melbourne-Roberto_intro} and Theorem \ref{thm:Caffarelli_intro} to establish majorization under log-concavity assumptions. Theorem \ref{thm:majorization_intro} improves on this result by relaxing to a log-subharmonicity assumption. 

So far one might be under the impression that in the context of majorization only the weaker determinant bound \eqref{eq:det_bound_main_intro}, as opposed to the stronger trace bound \eqref{eq:infty_bound_main_intro}, plays a role. This is not the case however as the following results show. As a first consequence of the trace bound \eqref{eq:infty_bound_main_intro} we show that we have monotonicity of majorization along Wasserstein geodesics.

\begin{theorem}[Monotonicity along Wasserstein geodesics]
\label{thm:mono_intro}
Let $\diff\source=e^{-\V}\diff x$ and $\diff\target=e^{-\W}\diff x$ be probability measures on $\R^{\dd}$, with $\source$ supported on all of $\R^{\dd}$,  such that there exist $\lsh>0,\lc>0$ with
\[
\Delta \V \le \lsh\dd \quad\text{and}\quad \nabla^2\W \succeq \lc \Id_{\dd}.
\]
Let $(\prob_t)_{t\in [0,1]}$ be the geodesic in Wasserstein space\footnote{The Wasserstein space is the space of probability measures on $\R^{\dd}$, with finite second moment, endowed with metric $(\source,\target)\mapsto \inf_{\textnormal{$\pi$ joint couplings of $\source$ and $\target$} } \int_{\R^{\dd}\times\R^{\dd}}|x-y|^2\diff\pi(x,y)$ \cite{villani2021topics}.} connecting $\source$ and $\target$. Then, if $\frac{\lsh}{\lc}\le 1$,
\begin{equation}
\label{eq:mono_geo}
[0,1]\ni t\mapsto \int_{\R^{\dd}}\cxfun(\prob_t(x))\diff x\textnormal{ is monotonically non-decreasing}
\end{equation}
for every convex function $\cxfun:\R_{\ge 0}\to \R$.
\end{theorem}
The second consequence of the trace bound \eqref{eq:infty_bound_main_intro} pertains to the important special case 
 of Theorem \ref{thm:majorization_intro} regarding differential entropy. 
 Letting $ \cxfun(x)=x\log x$ in the definition of majorization we get that, if $\frac{\lsh}{\lc}\le 1$, then the target measure must have a bigger differential entropy than the source measure,
\begin{equation}
\label{eq:entropy_domination}
\Ent(\target)\ge \Ent(\source),
\end{equation}
where
\[
\Ent(\prob):=\int_{\R^{\dd}} \prob\log\prob\quad \text{for absolutely continuous probability measures $\prob$ on $\R^{\dd}$}. 
\]
Our next result, which requires the trace bound \eqref{eq:infty_bound_main_intro}  as opposed to the weaker determinant bound  \eqref{eq:det_bound_main_intro}, provides a quantitative improvement of \eqref{eq:entropy_domination}.

\begin{theorem}[Stability of entropy domination]
\label{thm:entropy_stability_intro}
Let $\diff\source=e^{-\V}\diff x$ and $\diff\target=e^{-\W}\diff x$ be probability measures on $\R^{\dd}$, with $\source$ supported on all of $\R^{\dd}$,  such that there exist $\lsh>0,\lc>0$ with 
\[
\Delta \V \le \lsh\dd \quad\text{and}\quad \nabla^2\W \succeq \lc \Id_{\dd}.
\]
Let $\nabla\OT:\R^{\dd}\to\R^{\dd}$ be the Brenier map transporting $\source$ to $\target$.  If $\frac{\lsh}{\lc}\le  1$, then
\[
\Ent(\target)- \Ent(\source) \ge   \frac{1}{2\dd^2}\int_{\R^{\dd}}\|\nabla^2\OT-\Id_{\dd}\|_{\text{F}}^2\diff \source,
\]
where $\|\cdot\|_{\text{F}}$ is the Frobenius norm. In particular,
\begin{equation}
\label{eq:entropy_eq}
\Ent(\target)= \Ent(\source)\quad\quad\Longrightarrow \quad\quad \textnormal{$\target$ is a translate of $\source$}.
\end{equation}
\end{theorem}
Theorem \ref{thm:entropy_stability_intro} implies that if the differential entropies of $\source$ and $\target$ are close, then $\source$ and $\target$ themselves must be close, in the sense that the transport map $\nabla\OT$ between them must be close to the identity. In particular, the next corollary shows that if the entropies of $\source$ and $\target$ match then $\source=\target$.

\subsubsection{Volume contraction of the Kim-Milman map}
Our discussion so far focused on the Brenier map of optimal transport. While the regularity properties of the Brenier map are interesting in their own right, for the purpose of majorization any volume-contracting transport map would yield the same result (note however that this would not be the case for Theorem \ref{thm:mono_intro} and  Theorem \ref{thm:entropy_stability_intro}). The only other transport map for which regularity properties are known, in particular Lipschitz regularity, is the \emph{Kim-Milman heat flow map} \cite{MR2983070}. To define the Kim-Milman map  let $(\Pheat_t)_{t\ge 0}$ be the Langevin semigroup associated to $\target$, and define the family of diffeomorphisms $\KM_t:\R^{\dd}\to\R^{\dd}$ as the solution to the ordinary differential equation,
\begin{equation}
\label{eq:KM_construction_intro}
\partial_t\KM_t(x)=-\nabla\log\Pheat_t\density(\KM_t(x)),\quad  \KM_0(x)=x,\quad \text{for} ~ t>0 \text{ and } x\in\R^{\dd},
\end{equation}
where $\density:=\frac{\diff\source}{\diff\target}$. One can check that $\KM_t$ transports $\source$ to $\prob_t:=(\Pheat_t \density)\diff\target$. In particular, the Kim-Milman map $\KM:=\lim_{t\to\infty}\KM_t$ transports $\source$ to $\target$. Even assuming the existence of $\KM$ (which is unclear in our setting), the proof techniques that establish Theorem \ref{thm:volume-contraction_OT_intro} are not applicable for the Kim-Milman map. However, when  the target measure $\target$ is $\Gaussian:=\mathcal N(0,\Id_{\dd})$, the standard Gaussian measure on $\R^{\dd}$, we can in fact establish the analogue of \eqref{eq:det_bound_main_intro}. 

\begin{theorem_informal}
\label{thm:majorization_KM_intro}
Let $\diff\source=e^{-\V}\diff x$ be a probability measure on $\R^{\dd}$, with $\source$ supported on all of $\R^{\dd}$,  such that there exists $\lsh>0$ with
\[
\Delta \V(x) \le \lsh\dd, \quad\textnormal{for every $x\in\R^{\dd}$}.
\]
Then, under sufficient regularity, the Kim-Milman heat flow map $\KM$ between $\source$ and the standard Gaussian $\Gaussian$ on $\R^{\dd}$  satisfy
\begin{align}
\label{eq:det_bound_KM_intro}
 \|\det\nabla \KM\|_{L^{\infty}(\diff x)}\le \lsh^{\frac{\dd}{2}}.
 \end{align}
\end{theorem_informal}
We have purposely left Theorem \ref{thm:majorization_KM_intro} vague in terms of regularity, and we refer the reader to Section \ref{sec:KM}  below for a more precise discussion. Note however that in contrast to Theorem \ref{thm:main_OT_Lp_intro} where the convexity of the potentials ensures that a bound on the Laplacian implies a bound on the Hessian, which is helpful in approximation arguments (cf. Proposition \ref{prop:smooth}), in the case of the Kim-Milman map,  passing to the limit  in \eqref{eq:det_bound_KM_intro} seems to require stronger assumptions on the densities of $\source$ and $\target$.

\begin{remark}
\label{rem:compose}
For a source measure $\source$ and a target measure $\target$ the Kim-Milman heat flow map is based on the Langevin dynamics whose invariant measure is $\target$, starting the dynamics at $\source$. (In the original work of Kim and E. Milman \cite{MR2983070} the reverse flow is considered.) Theorem \ref{thm:majorization_KM_intro} only applies to the flow whose invariant measure $\target$ is equal to $\Gaussian$, the Gaussian measure. However, by composing two Kim-Milman maps we could get a $\left(\frac{\lsh}{\lc}\right)^{\frac{\dd}{2}}$-volume-contraction map between $\source$ and $\target$ under the assumptions of Theorem \ref{thm:main_OT_Lp_intro} (assuming sufficient regularity). Indeed, the  Kim-Milman  map between $\Gaussian$ and $\target=e^{-W}\diff x$, with $\nabla^2\W  \succeq \lc \Id_{\dd}$, is known to be $\frac{1}{\sqrt{\lc}}$-Lipschitz \cite[Theorem 1.1]{MR2983070}, \cite[Theorem 1]{MS2022}, so it suffices to compose this map with the Kim-Milman map of Theorem \ref{thm:majorization_KM_intro}.
\end{remark}

\subsection{Growth estimates in Fock spaces and log-subharmonic functions}
Let us now turn to the applications of majorization. We start by deriving some classical results on growth estimates of log-subharmonic functions. In particular, as mentioned above, the complex-analytic setting naturally gives rise to log-subharmonic functions.  Given $\sigma>0$ let $\Gaussian_{\sigma}:=\mathcal N(0,\sigma\Id_2)$ be the centered Gaussian measure with covariance $\sigma\Id_2$ on the complex plane $\C$.  For functions $\density:\C\to\R$, for which the following is finite, let
\[
\|f\|_{p,\sigma}:=\left(\frac{p}{2\pi\sigma}\int_{\C}\left|f(x)\Gaussian_{\sigma}(z)\right|^p\diff z\right)^{\frac{1}{p}},\quad p>0,
\]
and
\[
\|f\|_{\infty,\sigma}:=\textnormal{esssup}\left\{|f(z)|\Gaussian_{\sigma}(z):z\in \C\right\}.
\]
The following estimate is classical \cite[Theorem 2.7]{MR2934601}.
\begin{theorem}[Growth estimates in Fock space]
\label{thm:Fock}
Fix $0<p< \infty$ and $z\in \C$. Then, 
\begin{equation}
\label{eq:Fock}
|\density(z)|\le e^{\frac{|z|^2}{2\sigma}}\quad\textnormal{for all $\density$ which are entire and satisfy $\|\density\|_{p,\sigma}\le 1$}.
\end{equation}
\end{theorem}
\begin{proof}
Fix $0<p\le \infty$ and $z\in \C$. Let $f:\C\to\C$ be an entire function satisfying $\|\density\|_{p,\sigma}\le 1$. Define the probability measures $\source$ and $\target$ on $\C$ by
\begin{equation}
\label{eq:source_entire}
\diff \source(z):=\frac{|\density(z)|^p}{\|\density\|_{p,\sigma}^p} \Gaussian_{\frac{\sigma}{p}}(z)\diff z,\quad\text{and}\quad \diff\target(z):=\Gaussian_{\frac{\sigma}{p}}(z)\diff z.
\end{equation}
Since $|\density|$ is log-subharmonic we get that $\source$ and $\target$ satisfy the assumptions of Theorem \ref{thm:volume-contraction_OT_intro} with $\lsh=\lc=\frac{p}{\sigma}$. Hence, \eqref{eq:det_bound_main_intro}  yields that $|\det \nabla\OT(z)|\le 1$ for almost all $z\in \C$ where $\nabla\OT$ is the Brenier map between $\source$ and $\target$. By the Monge-Amp\`ere equation,
\begin{equation}
\label{eq:MA_Fock}
\frac{|\density(z)|^p}{\|\density\|_{p,\sigma}} \Gaussian_{\frac{\sigma}{p}}(z)= \Gaussian_{\frac{\sigma}{p}}(\nabla\OT(z))\det\nabla^2\OT(z),
\end{equation}
we get
\[
|\density(z)|^p \Gaussian_{\frac{\sigma}{p}}(z)\overset{\det\nabla^2\OT\le 1}{\le} \Gaussian_{\frac{\sigma}{p}}(\nabla\OT(z))\|\density\|_{p, \sigma}\overset{\|\density\|_{p,\sigma}\le 1}{\le}\Gaussian_{\frac{\sigma}{p}}(\nabla\OT(z))\overset{\Gaussian_{\frac{\sigma}{p}}\le \frac{p}{2\pi\sigma}}{\le}  \frac{p}{2\pi\sigma}.
\]
Hence,
\[
|\density(z)|^p\le e^{p\frac{|z|^2}{2\sigma}},
\]
which completes the proof.
\end{proof}
Inspecting the proof of Theorem \ref{thm:Fock} it is immediate that it is a statement about $\conv\dd$-log-subharmonic functions, where $\density:\R^{\dd}\to\R$ is $\conv\dd$-log-subharmonic, for $\conv\in \R$, if $x\mapsto \log f(x)-\conv\frac{|x|^2}{2}$ is subharmonic. 

\begin{theorem}[Growth estimates for log-subharmonic functions]
\label{thm:lsh_bound}
Fix $\beta\ge 0$ and let $\density:\R^{\dd}\to\R$ be a $(-\beta\dd)$-log-subharmonic function such that $\int_{\R^{\dd}}\density \diff\Gaussian=1$, where $\Gaussian$ is the standard Gaussian measure on $\R^{\dd}$. Then,
\begin{equation}
\label{eq:lsh_bound}
f(x)\le (\beta+1)^{\frac{\dd}{2}}e^{\frac{|x|^2}{2}}.
\end{equation}
\end{theorem}
\begin{proof}
Define the probability measures $\diff\source:=\density\diff\Gaussian$ and $\target:=\Gaussian$, and note that they satisfy the assumptions of Theorem \ref{thm:volume-contraction_OT_intro} with $\lsh=(\beta+1)$ and $\lc=1$. Hence, by \eqref{eq:det_bound_main_intro}, the Brenier map $\nabla\OT$ between $\source$ and $\target$ satisfies
\begin{equation}
\label{eq:lsh_det_bound_proof}
|\det\nabla^2\OT|\le (\beta+1)^{\frac{\dd}{2}}. 
\end{equation}
By the Monge-Amp\`ere equation,
\begin{align*}
\density(x)\Gaussian(x)= \Gaussian(\nabla\OT(z))\det\nabla^2\OT(z)\overset{\eqref{eq:lsh_det_bound_proof}}{\le }  \Gaussian(\nabla\OT(z))(\beta+1)^{\frac{\dd}{2}}\overset{\Gaussian\le (2\pi)^{-\frac{\dd}{2}}}{\le}  \left(\frac{(\beta+1)}{2\pi}\right)^{\frac{\dd}{2}},
\end{align*}
so
\[
f(x)\le (\beta+1)^{\frac{\dd}{2}}e^{\frac{|x|^2}{2}}.
\]
\end{proof}
In \cite[Lemma 2.1]{MR3723580} Theorem \ref{thm:lsh_bound} is obtained under the stronger log-convexity assumption. The log-subharmonic case was treated in \cite{MR2578455} (in the print version, see \cite[Remark 2.2]{MR3723580}), but only for $0$-log-subharmonic functions. Theorem \ref{thm:lsh_bound}, which is probably already known, generalizes to $\conv\dd$-log-subharmonic functions for any $\conv\le 0$.

\subsection{The Wehrl conjecture for Glauber states}
\label{subsec:Wehrl_intro}
Next we apply our majorization results to provide a new proof and stability results for the Wehrl conjecture for Glauber states. The Wehrl  entropy is a classical notion of entropy defined for quantum  mechanical states.  Wehrl presented this notion in \cite{wehrl1979relation} and conjectured that this entropy is  bounded by 1 (in dimension 1). The conjecture was proven by Lieb \cite{MR0506364}, and the characterization of the minimizers was established by Carlen \cite{MR1105661}. The validity of Wehrl-type conjectures in other settings has been actively investigated, but there are still questions which remain open  \cite{kulikov2022monotonicity, MR4580699}. In this section we propose a new transport approach towards this problem. In particular, we will provide another proof of the (generalized) Wehrl conjecture for Glauber states using volume-contracting transport maps, as well as stability results.

We will now present the setting of the original Wehrl conjecture following \cite{MR1105661}. Given a Schr\"odinger wave function $\wave\in L^2(\R^d,\diff x)$ define the \emph{coherent state transform} $\cotrans:L^2(\R^d,\diff x)\to L^2(\R^{2d},\diff q\diff p)$, where $\diff x,\diff q,\diff p$ denote the standard Lebesgue measure on $\R^{d}$, by
\begin{equation}
\label{eq:conerent_transform_def}
\cotrans \wave(q,p):=e^{i\pi \langle q,p\rangle}\int_{\R^{d}}e^{i \pi\langle x,p\rangle}e^{-\pi|x-q|^2}\wave(x)\diff x.
\end{equation}
The coherent state transform maps a wave function to a function on the phase space, whose modulus $|\cotrans \wave(q,p)|^2$ represents the density (not necessarily normalized) of phase space states. Up to a scaling the coherent state transform is isometric \cite[Proposition 3.4.1]{grochenig2001foundations}\footnote{The relation between the Bargmann transform  in \cite{grochenig2001foundations} and the coherent state transform is  $B \wave(q,p)=2^{-\frac{d}{4}}e^{-\frac{\pi}{2}(|q|^2+|p|^2)}\cotrans \wave(q,p)$.}, so if $|\psi|_{L^2(\R^{d},\diff x)}=2^{-\frac{d}{2}}$, then
\begin{equation}
\label{eq:phase_prob_def}
\prob_{\wave}(q,p):=|\cotrans \wave(q,p)|^2
\end{equation}
is a probability measure on the phase space $\R^{2d}$. Wehrl conjectured that the differential entropy of $\prob_{\wave}$ (the negative of the Wehrl entropy), satisfies the bound
\begin{equation}
\label{eq:Wehrl_conjecture}
\Ent(\prob_{\wave})=\int_{\R^{2d}}\log \prob_{\wave} \diff \prob_{\wave} \le -d.
\end{equation}
%The significance of this bound lies in the fact that $\Wehrle(\prob_{\wave})$ is nonnegative whenever $h\ge 1$. 
%\footnote{The quantum entropy is always nonnegative so when taking semiclassical limits to recover classical physics, we expect the classical entropy to also be nonnegative. However, the standard differential entropy can in fact be negative. The bound in the Wehrl conjecture shows that if entropy is defined in this manner then it is indeed nonnegative.}.
Equality is attained in \eqref{eq:Wehrl_conjecture}  for the \emph{Glauber states},
\begin{equation}
\label{eq:Gaussian_state}
\waveg(x):=e^{i\theta}\wave_{q_0,p_0}(x),
\end{equation}
where $\theta \in \R/ 2\pi\mathbb Z $, $q_0,p_0\in	\R^{d}$, and
\[
\wave_{q_0,p_0}(x):=2^{\frac{d}{4}}e^{2\pi i\langle x,p_0\rangle}e^{-\pi|x-q_0|^2}.
\]
In other words, according to the conjecture  Glauber states maximize the differential entropy (equivalently minimize the Wehrl entropy),
\begin{equation}
\label{eq:Wehrl_conjecture_ent}
\Ent(\prob_{\wave})\le \Ent(\prob_{\waveg})\quad\text{for all wave functions }\wave\in L^2(\R^d,\diff x).
\end{equation}
As mentioned above, the conjectured inequality \eqref{eq:Wehrl_conjecture_ent} and the characterization of its equality cases (only Glauber states are minimizers) were proven in \cite{MR0506364, MR1105661}. In fact, a stronger statement can be made, namely the \emph{generalized Wehrl conjecture for Glauber states}, where the entropy is replaced by other convex functions, which in our terminology means that Glauber states majorize; see  \cite{kulikov2022monotonicity, MR4580699} for proofs of the inequalities and \cite{frank2023generalized} for their stability.  Our next result applies the tools developed in the previous sections to provide a new proof and stability results for the generalized Wehrl conjecture for Glauber states. In fact, we can also address \emph{mixed states}, 
\begin{equation}
\label{eq:mixed}
\prob_{\wave_1,\ldots,\wave_k}(q,p):=\sum_{j=1}^k\lambda_j|\cotrans \wave_j(q,p)|^2,
\end{equation}
where $\{\lambda_j\}_{j=1}^k$ are nonnegative weights which sum up to 1, and $\{\wave_j\}_{j=1}^k$ forms an orthogonal system in $L^2(\R^d,\diff x)$ with all individual square norms equal to $2^{-\frac{d}{2}}$. (The measure constructed in \eqref{eq:phase_prob_def} is a \emph{pure} states where all the weights but one vanish.) Our approach treats mixed and pure states in the same way which simplifies the analysis of stability, which in general is more difficult for mixed states \cite{frank2023generalized}.

\begin{theorem}[Generalized Wehrl conjecture for Glauber states]
\label{thm:generalized_Wehrl_conjecture} Let $k$ be a positive integer, let $\{\wave_j\}_{j=1}^k$  be an orthogonal system in $L^2(\R^d,\diff x)$ with all individual square norms equal to $2^{-\frac{d}{2}}$, and let $\{\lambda_j\}_{j=1}^k$ be nonnegative weights which sum up to 1. Define the probability measure
\[
\prob_{\wave_1,\ldots,\wave_k}(q,p):=\sum_{j=1}^k\lambda_j|\cotrans \wave_j(q,p)|^2,
\]  
and let $\nabla\OT$ be the Brenier map between $\prob_{\wave_1,\ldots,\wave_k}$ and $\prob_{\waveg}$. Then,
\begin{align}
\label{eq:infty_bound_Wehrl}
\|\Delta\OT\|_{L^{\infty}(\diff x)}\le 2d.
\end{align}
Moreover,
\begin{align}
\label{eq:majorization_Wehrl}
\int_{\R^{2d}} \cxfun(\prob_{\wave_1,\ldots,\wave_k})\diff p\diff q\le \int_{\R^{2d}} \cxfun(\prob_{\waveg})\diff p\diff q\quad\text{for every convex function }\mathrm \cxfun:[0,1] \to \R,
\end{align}
\begin{equation}
\label{eq:mono_geo_Wehrl}
[0,1]\ni t\mapsto \int_{\R^{\dd}}\cxfun(\prob_t(x))\diff x\textnormal{ is monotonically non-decreasing},
\end{equation}
where  $(\prob_t)_{t\in [0,1]}$ is the geodesic in Wasserstein space connecting $\prob_{\wave_1,\ldots,\wave_k}$ and $\prob_{\waveg}$, $\prob_t:=(\nabla\OT_t)_{\sharp}\prob_{\wave_1,\ldots,\wave_k}$ where $\nabla\OT_t(x):=(1-t)x+t\nabla\OT(x)$. Finally,
\begin{align}
\label{eq:entropy_stability_Wehrl}
\Ent(\prob_{\waveg})- \Ent(\prob_{\wave_1,\ldots,\wave_k}) \ge   \frac{1}{8d^2}\int_{\R^{2d}}\|\nabla^2\OT-\Id_{2d}\|_{\text{F}}^2\,\diff \prob_{\wave_1,\ldots,\wave_k}.
\end{align}
\end{theorem}

\begin{proof}
First note that from the definition of $\prob_{\wave_1,\ldots,\wave_k}$ and the assumptions on $\{\wave_j\}_{j=1}^k$, we have that $\prob_{\wave_1,\ldots,\wave_k}\le 1$. This will allow us later to restrict  to convex functions $\cxfun:[0,1]\to \R$ in \eqref{eq:majorization_Wehrl}. To prove \eqref{eq:infty_bound_Wehrl} first note that direct computation shows that $\prob_{\waveg}$ is a Gaussian on $\R^{2d}$ with mean $(q_0,-p_0)$ and covariance $\frac{1}{2\pi}\Id_{2d}$,
\[
\prob_{\waveg}=\mathcal N\left((q_0,-p_0),\frac{1}{2\pi}\Id_{2d}\right).
\]
On the other hand, the relation between the coherent state transform and the Bargmann transform \cite[Proposition 3.4.1]{grochenig2001foundations}\footnotemark[\value{footnote}] yields, for each $j=1,\ldots,k$,
\begin{equation}
\label{eq:STFT_Bargmann_j}
|\cotrans\wave_j(q,p)|^2=2^{-\frac{d}{2}}\left|\tilde f_j(q+ip)\right|^2e^{-\pi(|q|^2+|p|^2)},
\end{equation}
where $\tilde f_j:\mathbb C^{d}\to \mathbb C^{d}$ is an entire function. In particular,
\begin{equation}
\label{eq:STFT_Bargmann}
\prob_{\wave_1,\ldots,\wave_k}(q,p)=2^{-\frac{d}{2}}e^{-\pi(|q|^2+|p|^2)}\sum_{j=1}^k\lambda_j\left|\tilde f_j(q+ip)\right|^2.
\end{equation}
The logarithm of the modulus of an entire function is subharmonic, so for each $j=1,\ldots, k$, $|\tilde f_j(q+ip)|^2$ is log- subharmonic. Since the  mixture of log-subharmonic functions is log-subharmonic \cite[Proposition 2.2]{MR2578455}, it follows that $\sum_{j=1}^k\lambda_j\left|\tilde f_j(q+ip)\right|^2$ is log-subharmonic. Hence, we can apply Theorem \ref{thm:main_OT_Lp_intro} and Theorem \ref{thm:volume-contraction_OT_intro}, (in fact also Theorem \ref{thm:majorization_KM_intro}), with 
\[
\dd:=2d,\quad \source:=\prob_{\wave_1,\ldots,\wave_k},\quad\target:=\prob_{\waveg},\quad \lc:=2\pi,\quad \lsh:=2\pi.
\]
In particular, $\sqrt{ \frac{\lsh}{\lc}}=1$ which yields the bound \eqref{eq:infty_bound_Wehrl} and also means that we can apply Theorem \ref{thm:majorization_intro}, Theorem \ref{thm:mono_intro}, and Theorem \ref{thm:entropy_stability_intro}  to get \eqref{eq:majorization_Wehrl}, \eqref{eq:mono_geo_Wehrl}, and \eqref{eq:entropy_stability_Wehrl}, respectively.
\end{proof}

\subsection{Two-dimensional Coulomb gases}
\label{subsec:plasma_intro} 
Coulomb gases are important models in mathematical physics with deep connections to random matrix theory; see \cite{MR4619310} for a brief survey. In this section we review the two-dimensional (one-component) Coulomb gas, to which our results above will apply. We consider $N$ particles in $\R^2$ of identical charge in  a fixed neutralizing background  at inverse temperature $\beta$, whose interactions are logarithmic. We identify $\R^2 \simeq\C$ and denote the positions of the particles as $z:=(z_1,\ldots,z_N)\in\C^N$, for $z_i\in \C$. Let $\source_{\beta,N}$  be the probability measure on $\C^{N}$ defined by
\begin{equation}
\label{eq:mu_betaN}
\diff \source_{\beta,N,Q}(z):=\frac{e^{-Q_{\beta,N}(z)-G_{\beta,N}(z)}}{\mathcal Z_{\beta, N}}\diff  z\quad \text{with}\quad \mathcal Z_{\beta, N}:=\int_{\C^N}e^{-Q_{\beta,N}(z)-G_{\beta,N}(z)}\diff z,
\end{equation}
where
\begin{equation}
\label{eq:V_betaN}
Q_{\beta,N}(z):=\beta N\sum_{j=1}^NQ(z_j)\quad\text{with}\quad Q:\C\to \R,
\end{equation}
and
\begin{equation}
\label{eq:V_betaN}
G_{\beta,N}(z):=-\beta \sum_{i<j=1}^N\log |z_i-z_j|.
\end{equation}
The term $Q_{\beta,N}$ models the particle-background interaction, and the term $G_{\beta,N}$  models the particle-particle logarithmic interaction. A common choice for $Q_{\beta,N}$ is taking $Q$ to be quadratic, $Q(z_j)=\frac{|z_j|^2}{2}$. Since $\Delta  G_{\beta,N}\le 0$, we can apply our results to the  two-dimensional (one-component) Coulomb gas $\source_{\beta,N,Q}$ and the pure background model $\target_{\beta,N,Q}$ given by 
\[
\diff \target_{\beta,N,Q}(z):=\frac{e^{-Q_{\beta,N}(z)}}{\int e^{-Q_{\beta,N}(z)}\diff z}\diff z.
\]
\begin{theorem}[Two-dimensional one-component gas]
\label{thm:Coulomb}
Suppose there exists $\lsh_2,\lc_2>0$ such that 
\[
\Delta Q \le 2\lsh_2\quad\text{and}\quad \nabla^2Q \succeq \lc_2 \Id_{2},
\]
where $\Delta Q, \nabla^2Q$ are the distributional Laplacian and Hessian on $\R^2$, respectively. Let $\nabla\OT_{\beta,N,Q}$ be the optimal transport map between $\source_{\beta,N,Q}$ and $\target_{\beta,N,Q}$. Then, 
\begin{align}
\label{eq:infty_bound_plasma}
\|\Delta\OT_{\beta,N,Q}\|_{L^{\infty}(\diff x)}\le 2N\sqrt{\frac{\lsh_2}{\lc_2}} ,\quad\text{and in particular}\quad  \|\det\nabla^2\OT_{\beta,N,Q}\|_{L^{\infty}(\diff x)}\le \left(\frac{\lsh_2}{\lc_2}\right)^N.
\end{align}
Moreover, when 
\begin{equation}
\label{eq:quad_coulomb}
Q(z_j)=\frac{|z_j|^2}{4}\quad\textnormal{ for all }\quad j=1,\ldots, N,
\end{equation}
we have
\begin{align}
\label{eq:majorization_plasma}
\int_{\C^N} \cxfun(\source_{\beta,N,Q})\diff z\le \int_{\C^N} \cxfun(\target_{\beta,N,Q})\diff z\quad\text{for every convex function }\mathrm \cxfun:\R_{\ge 0} \to \R,
\end{align}
\begin{equation}
\label{eq:mono_geo_plasma}
[0,1]\ni t\mapsto \int_{\R^{\dd}}\cxfun(\prob_{t,\beta,N,Q}(x))\diff x\textnormal{ is monotonically non-decreasing},
\end{equation}
where  $(\prob_{t,\beta,N,Q})_{t\in [0,1]}$ is the geodesic in Wasserstein space connecting $\source_{\beta,N,Q}$ and $\target_{\beta,N,Q}$, $\prob_{t,\beta,N,Q}:=(\nabla\OT_{t,\beta,N,Q})_{\sharp}\source_{\beta,N,Q}$ where $\nabla\OT_{t,\beta,N,Q}(x):=(1-t)x+t\nabla\OT_{\beta,N,Q}(x)$, and
\begin{align}
\label{eq:entropy_stability_plasma}
\Ent(\target_{\beta,N,Q})- \Ent(\source_{\beta,N,Q}) \ge   \frac{1}{8N^2}\int_{\C^N}\|\nabla^2\OT_{\beta,N,Q}-\Id_{2N}\|_{\text{F}}^2\,\diff \source_{\beta,N,Q}.
\end{align}
\end{theorem}
\begin{proof}
Denote $\diff\source_{\beta,N,Q}=:e^{-\V_{\beta,N,Q}}\diff z$. The assumption $\nabla^2Q  \succeq \lc_2 \Id_{2}$ implies that $\nabla^2Q_{\beta,N}  \succeq \lc_2 \beta N\Id_{2N}$. Since $\Delta  G_{\beta,N}\le 0$, the assumption $\Delta Q \le 2\lsh_2$ implies $\Delta \V_{\beta,N,Q}=\Delta Q_{\beta,N}+\Delta  G_{\beta,N}\le 2\lsh_2\beta N^2$.  Hence, setting 
\[
\dd:=2N,\quad \source:=\source_{\beta,N,Q},\quad\target:=\target_{\beta,N,Q},\quad \lc:=\lc_2\beta N,\quad \lsh:=\lsh_2\beta N,
\]
we can apply Theorem \ref{thm:main_OT_Lp_intro}  and Theorem \ref{thm:volume-contraction_OT_intro} to get \eqref{eq:infty_bound_plasma}. The remaining assertions of the theorem hold only when $Q$ is quadratic,  where $\lsh_2=\lc_2=1$, by Theorem \ref{thm:majorization_intro}, Theorem \ref{thm:mono_intro}, and Theorem \ref{thm:entropy_stability_intro}, respectively.
\end{proof}

\subsection*{Organization of paper}
Section \ref{sec:divergence_OT} proves  Theorem \ref{thm:main_OT_Lp_intro}, while Section \ref{sec:major} focuses on majorization, proving in  particular Theorem \ref{thm:mono_intro}  and Theorem \ref{thm:entropy_stability_intro}. Section \ref{sec:KM} is dedicated to the Kim-Milman heat flow map. Section \ref{sec:appendix} is an appendix containing a number of technical results. 
\subsection*{Acknowledgments }
We thank Alessio Figalli and Cole Graham for useful conversations. We thank Rupert Frank for explaining to us the mixed state case of the Wehrl conjecture, Emanuel Milman for Remark \ref{rem:compose}, Ramon van Handel for pointing the applicability of our results to two-dimensional Coulomb gases, and Ammari Bader for pointing out a mistake in a previous version of Theorem \ref{thm:Coulomb}. We also thank the anonymous referee for their helpful comments. 

This material is based upon work supported by the National Science Foundation under Award Number DMS-2331920 and DMS 2055686. G.D.P. also acknowledges the support of the Simons Foundation.

\section{The divergence of the Brenier map}
\label{sec:divergence_OT}  
In this section we will prove $L^p$ estimates on $\Delta\OT$ (Theorem \ref{thm:main_OT_Lp}), which in particular will imply Theorem \ref{thm:main_OT_Lp_intro}. While the proof of our main result below Theorem \ref{thm:main_OT_Lp} is quite technical, at the formal level it follows the original (formal) derivation of Caffarelli \cite{MR1800860}, as we now explain. Our goal is to derive the bound 
\begin{equation}
\label{eq:infty_bound_main_sec}
\|\Delta\OT\|_{L^{\infty}(\diff x)}\le \dd\sqrt{ \frac{\lsh}{\lc}} \quad \text{when}\quad \Delta \V\le \lsh\dd ~\text{and}~ \nabla^2\W\succeq \lc \Id_{\dd}.
\end{equation}
The derivation is based on differentiating the Monge-Amp\`ere equation twice, and using the optimality conditions at the point where $\Delta\OT$ attains its maximum. To this end, given a unit vector $e\in \R^{\dd}$ and a function $\xi:\R^{\dd}\to \R$, denote by $\xi_e$ (res. $\xi_{ee}$) the first (res. second) directional derivative of $\xi$ in the direction $e$. The Monge-Amp\`ere equation reads
\begin{equation}
\label{eq:Monge}
e^{-\V}=e^{-\W}(\nabla \OT)\det\nabla^2\OT,
\end{equation}
so take the logarithm on both sides of \eqref{eq:Monge},  and differentiate twice in a fixed direction $e$, to get
\begin{align}
\label{eq:Vee}
\V_{ee}&=\langle \nabla^2\W(\nabla\OT)\nabla^2\OT e,\nabla^2\OT e\rangle+\langle \nabla\W,\nabla\OT_{ee}\rangle-\partial_{ee}\log\det \nabla^2\OT.
\end{align}
Using the identity
\begin{equation}
\label{eq:Jacobi_ee}
\partial_{ee}\log\det \nabla^2\OT=\Tr[( \nabla^2\OT)^{-1}\nabla^2\OT_{ee}]-\Tr\left[\left(( \nabla^2\OT)^{-1}\nabla^2\OT_e\right)^2\right],
\end{equation}
the identity \eqref{eq:Vee} becomes
\begin{align}
\label{eq:Vee+}
\V_{ee}&=\langle \nabla^2\W(\nabla\OT)\nabla^2\OT e,\nabla^2\OT e\rangle+\langle \nabla\W,\nabla\OT_{ee}\rangle-\Tr[( \nabla^2\OT)^{-1}\nabla^2\OT_{ee}]+\Tr\left[\left(( \nabla^2\OT)^{-1}\nabla^2\OT_e\right)^2\right].
\end{align}
Using $\nabla^2\W\succeq \lc \Id_{\dd}$, \eqref{eq:Vee+} implies
\begin{align}
\label{eq:Vee_inq}
\V_{ee}\ge \lc|\nabla^2\OT e|^2+\langle \nabla\W,\nabla\OT_{ee}\rangle-\Tr[( \nabla^2\OT)^{-1}\nabla^2\OT_{ee}]+\Tr\left[\left(( \nabla^2\OT)^{-1}\nabla^2\OT_e\right)^2\right],
\end{align}
so summing on both sides of \eqref{eq:Vee_inq} over a basis $\{e_i\}$ of $\R^{\dd}$ yields 
\begin{equation}
\label{eq:Jacobi_Laplacian}
\Delta \V\ge\lc \sum_{i=1}^{\dd}|\nabla^2\OT e_i|^2+ \langle \nabla\W,\nabla\Delta\OT\rangle-\Tr[( \nabla^2\OT)^{-1}\nabla^2\Delta\OT]+\sum_{i=1}^{\dd}\Tr\left[\left(( \nabla^2\OT)^{-1}\nabla^2\OT_{e_i}\right)^2\right].
\end{equation}

Suppose $\Delta\OT$ attains its maximum $x_0$.  Then  the optimality conditions give $\nabla \Delta\OT(x_0)=0$ and $\nabla^2\Delta\OT(x_0) \preceq 0$, so
\begin{equation}
\label{eq:opt_condition}
\langle \nabla\W(x_0),\nabla\Delta\OT(x_0)\rangle-\Tr[( \nabla^2\OT(x_0))^{-1}\nabla^2\Delta\OT(x_0)]\ge 0.
\end{equation}
On the other hand,
\begin{equation}
\label{eq:Tr_noneg}
\Tr\left[\left(( \nabla^2\OT)^{-1}\nabla^2\OT_e\right)^2\right]=\Tr[A^2]\ge 0\quad\text{with}\quad A:=( \nabla^2\OT)^{-\frac{1}{2}}\nabla^2\OT_e( \nabla^2\OT)^{-\frac{1}{2}}.
\end{equation}
Hence, combining \eqref{eq:Jacobi_Laplacian}, \eqref{eq:opt_condition}, \eqref{eq:Tr_noneg}, and applying the Cauchy-Schwarz inequality, shows that 
\begin{equation}
\label{eq:main_inq_proof_e}
\lsh\dd\ge\Delta \V(x_0)\ge  \lc \sum_{i=1}^{\dd}|\nabla^2\OT(x_0) e_i|^2\ge \kappa\sum_{i=1}^{\dd}\langle \nabla^2\OT(x_0)e_i,e_i\rangle^2. 
\end{equation}
By Jensen's inequality,
\begin{equation}
\label{eq:main_inq_proof}
\lsh\dd\ge\Delta \V(x_0)\ge \lc\sum_{i=1}^{\dd}\langle \nabla^2\OT(x_0)e_i,e_i\rangle^2\ge \frac{\lc}{\dd}[\Delta\OT(x_0)]^2,
\end{equation}
so it follows that
\begin{equation}
\label{eq:OT_subharmonic}
\Delta\OT(x_0)\le \dd\sqrt{\frac{\lsh}{\lc}},
\end{equation}
where we used that $\Delta\OT(x_0)\ge 0$ since $\OT$ is convex. Since $x_0$ is a point where $\Delta\OT$ attains its maximum, we conclude \eqref{eq:infty_bound_main_sec}. Making the above argument rigorous is difficult due to the need to show the existence of a point $x_0$ where $\Delta\OT$ attains its maximum. Instead, we follow the $L^p$ approach of Kolesnikov \cite{MR3201654}, \cite[\S 6]{kolesnikov2011mass}, but with some necessary modifications (cf. Remark \ref{rem:modification}).

\begin{theorem}
\label{thm:main_OT_Lp}
Let $\diff\source=e^{-\V}\diff x$ and $\diff\target=e^{-\W}\diff x$ be probability measures on $\R^{\dd}$, with $\source$ supported on all of $\R^{\dd}$,  such that there exist $\lsh>0,\lc>0$ with
\[
\Delta \V \le \lsh\dd \quad\text{and}\quad \nabla^2\W \succeq \lc \Id_{\dd}.
\]
 Let $\nabla\OT:\R^{\dd}\to\R^{\dd}$ be the Brenier map transporting $\source$ to $\target$. Then,
\begin{align}
\label{eq:infty_bound_main}
\|\Delta\OT\|_{L^{\infty}(\diff x)}\le \dd\sqrt{ \frac{\lsh}{\lc}}.
\end{align}
\end{theorem} 

We start by showing that it suffices to prove Theorem \ref{thm:main_OT_Lp} assuming sufficient regularity for $\source$ and $\target$. These regularity assumptions are captured as follows.

\begin{assumption}
\label{assump}
$~$

\begin{enumerate}
\item $\V$ is smooth.
\item There exists a constant $c>0$ such that $\nabla^2\V(x) \preceq c\Id_{\dd}$ for all $x\in \R^{\dd}$.
\item There exist constants $a,b >0$ such that $\lim_{|x|\to\infty}\V(x)-(b|x|^2-a)\ge 0$. 
\item $\W$ is smooth.
\item There exists a constant $d>0$ such that $\nabla^2\W(x) \preceq d\Id_{\dd}$ for all $x\in \R^{\dd}$.
\end{enumerate}
\end{assumption}
Note that Assumption \ref{assump}(3) implies that $\source$ has a finite second moment while  the assumption  $\nabla^2\W  \succeq \lc \Id_{\dd}$ implies that $\target$ has a finite second moment (by the Poincar\'e inequality for $\target$).

\begin{proposition}
\label{prop:smooth}
It suffices to prove Theorem \ref{thm:main_OT_Lp} for $\source,\target$ which in addition to satisfying the assumptions of Theorem \ref{thm:main_OT_Lp} also satisfy Assumption \ref{assump}. 
\end{proposition}

\begin{proof}
The proof is based on the composition of the following two steps. 

\noindent\textbf{Step 1.} Given $\source,\target$ which satisfy the assumptions of Theorem \ref{thm:main_OT_Lp} we show that there exist sequences of probability measures $\{\source_k\},\{\target_k\}$ converging weakly to $\source,\target$, respectively, such that,  for all $k$,  each pair $\source_k,\target_k$  satisfies the assumptions of Theorem \ref{thm:main_OT_Lp} and Assumption \ref{assump}.

\noindent\textbf{Step 2.} Assume that Theorem \ref{thm:main_OT_Lp} holds true for each pair $\source_k,\target_k$ and then pass to the limit to deduce \eqref{eq:infty_bound_main}.

Before we start with step 1 let us state a number of important smoothing properties of the  Ornstein-Uhlenbeck semigroup which we will use to smooth out the measures $\source$ and $\target$.

\begin{proposition}
\label{prop:OU_smooth}
Let $\Gaussian$ be the standard Gaussian measure in $\R^{\dd}$, and let $f:\R^{\dd}\to \R_{\ge 0}$ be a nonnegative function in $L^1(\Gaussian)$. Let $(\Pheat_t)_{t\ge 0}$ be the Ornstein-Uhlenbeck semigroup,
\begin{equation}
\label{eq:OU_def_prop}
\Pheat_t\density(x):=\int_{\R^{\dd}} \density(e^{-t}x+\sqrt{1-e^{-2t}}y)\diff \Gaussian(y),\quad  t\ge 0,\quad x\in \R^{\dd}.
\end{equation}
\begin{enumerate}[(i)]
\item For every $x\in \R^{\dd}$ and $t>0$,
\begin{equation}
\label{eq:less_lc}
\nabla^2\log\Pheat_t\density(x) \succeq -\frac{e^{-2t}}{1-e^{-2t}} \Id_{\dd}\quad \textnormal{for all }x\in\R^{\dd} \text{ and }t>0.
\end{equation}
\item Suppose $\density$ is $\conv$-log-concave for $\conv\in \R$  (i.e., $x\mapsto \log \density(x)-\conv\frac{|x|^2}{2}$ is concave). Then, for every $x\in \R^{\dd}$,
\begin{equation}
\label{eq:lc_OU}
\nabla^2\log\Pheat_t\density(x)\preceq \frac{ e^{-2t}\conv}{1-\conv(1-e^{-2t})} \Id_{\dd}\begin{cases}
\text{for any }t\in [0,\infty) &\text{if }\conv\le 1\\
\text{for any }t\in \left[0,\log\left(\sqrt{\frac{\conv}{\conv-1}}\right)\right] &\text{if }\conv >1.
\end{cases}
\end{equation}
\item Suppose $\density$ is $\conv\dd$-log-subharmonic for some $\conv\le 0$ (i.e.,  $x\mapsto \log f(x)-\conv\frac{|x|^2}{2}$ is subharmonic). Then, for every $x\in \R^{\dd}$ and $t>0$,
\begin{equation}
\label{eq:beta_subharmonic}
\Delta\log\Pheat_t\density(x)\ge \frac{e^{-2t}\conv\dd}{1-\conv(1-e^{-2t})}\ge \conv\dd.
\end{equation}
\end{enumerate}
\end{proposition}
The proof of Proposition \ref{prop:OU_smooth} (with some extra results) is given in Section \ref{subsec:appendix_smooth}.

Let us begin with \textbf{Step 1}. We first construct the sequence $\{\source_k\propto e^{-\V_k}\diff x\}$ whose members all satisfy the assumptions of Theorem \ref{thm:main_OT_Lp}  as well as Assumption \ref{assump}(1-3). Given $\source=e^{-\V}\diff x$ let 
\[
\tilde V_k:=-\log\Pheat_{\frac{1}{k}}e^{-\V}.
\]
Then $\tilde V_k$ is smooth, and by Proposition \ref{prop:OU_smooth}(iii) it satisfies $\Delta\tilde\V_k(x)\le \lsh\dd$ for all $x\in \R^{\dd}$. This shows that we can construct a measure $\propto e^{-\tilde \V_k}\diff x$ which satisfies the assumptions of Theorem \ref{thm:main_OT_Lp} as well as Assumption \ref{assump}(1). Moreover, Assumption \ref{assump}(2) is also satisfied for some $c_k$ by  Proposition \ref{prop:OU_smooth}(i). Next we modify $\tilde V_k$ so that  Assumption \ref{assump}(3) is also satisfied. Define 
\[
\V_k(x):=\left(1-\frac{1}{k}\right)\tilde \V_k(x)+\frac{1}{k}\lsh\frac{|x|^2}{2}
\]
and note that  $\Delta\V_k\le \lsh\dd$,  $\V_k$ is smooth, and 
\[
\nabla^2\V_k \preceq \left(\left(1-\frac{1}{k}\right)c_k+\frac{1}{k}\lsh\right)\Id_{\dd},
\]
so the measure $\propto e^{- \V_k}\diff x$ satisfies the assumptions of Theorem \ref{thm:main_OT_Lp} as well as Assumptions \ref{assump}(1-2). Let us show that $\V_k$ also satisfies Assumption \ref{assump}(3). To this end we first need to argue that $\V$ is positive outside of some ball. Indeed, since $\Delta\V\le \lsh\dd$, we have that, for any $x_0\in \R^{\dd}$, the function 
\[
\R^{\dd}\ni x\mapsto e^{\lsh\frac{|x-x_0|^2}{2}-\V(x)}
\]
is subharmonic. Hence, choosing $x_0$ such that $|x_0|>R$ for some $R>0$, we have
\begin{equation}
\label{eq:V_bd}
e^{-\V(x_0)}\le \avgint_{B_1(x_0)}e^{\lsh\frac{|x-x_0|^2}{2}-\V(x)}\diff x\le \frac{e^{\frac{\lsh}{2}}}{\Vol(B_1)}\int_{(B_{R-1}(0))^c}e^{-\V(x)}\diff x.
\end{equation}
The right-hand side of \eqref{eq:V_bd} converges to $0$ as $R\to\infty$ since $\int e^{-V}$ is finite. Hence, $\lim_{x\to\infty}\V(x)=+\infty$ and  Assumption \ref{assump}(3) follows by the construction of $\V_k$. Finally, let $\source_k \propto e^{-\V_k}\diff x$ and note that $\source_k\to \source$ weakly.

Next we construct the sequence $\{\target_k\propto e^{-\W_k}\diff x\}$ whose members all satisfy the assumptions of Theorem \ref{thm:main_OT_Lp}  as well as Assumption \ref{assump}(4-5). Given $\target=e^{-\W}\diff x$  let 
\[
\W_k:=-\log\Pheat_{\frac{1}{k}}e^{-\W}
\]
and let $\{\target_k\propto e^{-\W_k}\diff x\}$. Then, as in the construction of $\V_k$, $\W_k$ is smooth and there exists a constant $d_k$ such that $\nabla^2\W _k\preceq d_k\Id_{\dd}$. Hence, $\W_k$ satisfies Assumption \ref{assump}(4-5). Finally, by Proposition \ref{prop:OU_smooth}(ii), $\W_k$ satisfies the assumptions of Theorem \ref{thm:main_OT_Lp} with $\nabla^2\W_k \succeq \lc_k$ where 
\begin{equation}
\label{eq:lc_k}
\lc_k:= \frac{\lc e^{-\frac{2}{k}}}{1+\lc(1-e^{-\frac{2}{k}})} =\lc-\frac{\lc(\lc +1)(1-e^{-\frac{2}{k}})}{1+\lc(1-e^{-\frac{2}{k}})}.
\end{equation}

We now move to \textbf{Step 2}. Assume that Theorem \ref{thm:main_OT_Lp} holds true for each pair $\source_k,\target_k$, and let $\nabla\OT_k$ be the Brenier map between $\source_k$ and $\target_k$.  Then, for $k$ large enough, by Theorem \ref{thm:main_OT_Lp} and Remark \ref{rem:Lipschiz},
\begin{equation}
\label{eq:Delta_Phi_k}
\|\Delta\OT_k\|_{L^{\infty}(\diff x)}\le \dd\sqrt{ \frac{\lsh}{\lc_k}}\quad\text{and}\quad \nabla^2\OT_k(x) \preceq \dd\sqrt{ \frac{\lsh}{\lc_k}}\Id_{\dd} \text{ for every $x\in \R^{\dd}$}.
\end{equation}
By \cite[Lemma 1]{MS2022}, which follows \cite[Lemma 2.1]{neeman2022lipschitz} building on \cite[Lemma 3.3]{MR2983070}, we get that, up to a subsequence, $\{\nabla\OT_k\}$ converges  almost everywhere to some transport map $T$ between  $\source$ and $\target$.
Since $\{\nabla\OT_k\}$ converges so does $\{\OT_k\}$ to some convex function $\OT$. It follows that $T=\nabla\OT$ is the Brenier map between $\source$ and $\target$. Finally, the proof is complete  since \(\Phi_{k}\) converges to \(\Phi\) locally uniformly, \(\Delta \Phi_{k}\) converges to \(\Delta \Phi\) as distribution, and as
\[
\begin{split}
\|\Delta \Phi\|_{L^{\infty}(\diff x)}&=\sup_{\substack{\eta \in C_{c}^{\infty},\|\eta\|_{L^1}=1}} \int \Delta \Phi(x) \eta(x) \diff x =\sup_{\substack{\eta \in C_{c}^{\infty}, \|\eta\|_{L^1}=1}} \int  \Phi(x) \Delta \eta(x) \diff x
\\
&=\sup_{\substack{\eta \in C_{c}^{\infty}, \|\eta\|_{L^1}=1}} \lim_{k \to \infty}  \int  \Phi_{k}(x) \Delta \eta(x) \diff x =\sup_{\eta \in C_{c}^{\infty}, \|\eta\|_{L^{}(\diff x)}=1} \lim_{k \to \infty} \int  \Delta \Phi_{k}(x)  \eta(x) \diff x
\\
&\le \lim_{k\to \infty} \|\Delta \Phi_{k}\|_{L^{\infty}(\diff x)}=\dd\sqrt{ \frac{\lsh}{\lc}}.
\end{split}
\]
\end{proof}

\begin{proof}[Proof of Theorem \ref{thm:main_OT_Lp}]
In light of Proposition \ref{prop:smooth} we may assume Assumption \ref{assump} from here on. Under the assumptions of Theorem \ref{thm:main_OT_Lp}, together with Assumption \ref{assump}, we have that there exists an optimal transport map $\nabla\OT$ between $\source$ and $\target$ which is smooth \cite[Theorem 1.1]{Dario2019}, satisfies the Monge-Amp\`ere equation,
\begin{equation}
\label{eq:Monge_proof}
e^{-\V(x)}=e^{-\W(\nabla \OT(x))}\det\nabla^2\OT(x), \quad \text{for all }x\in \R^{\dd},
\end{equation}
and, by Theorem \ref{thm:Caffarelli_intro},  satisfies the bound
\begin{equation}
\label{eq:Caffarelli_bound}
0\preceq\nabla^2\OT(x)\preceq C\Id_{\dd}\quad \text{for all }x\in \R^{\dd},
\end{equation}
for some constant $C>0$.
Taking the logarithm in \eqref{eq:Monge_proof} we get
\begin{align}
\label{eq:Monge_log}
\V(x)=\W(\nabla\OT(x))-\log \det \nabla^2\OT(x),
\end{align}
so, for every $x,y\in \R^{\dd}$,
\begin{align}
\begin{split}
V(x+y)+V(x-y)-2V(x)&=\left\{\W(\nabla\OT(x+y))+\W(\nabla\OT(x-y))-2\W(\nabla\OT(x))\right\}\label{eq:Monge_finite_diff}\\
&-\log\left[\frac{\det \nabla^2\OT(x+y)\det \nabla^2\OT(x-y)}{(\det \nabla^2\OT(x))^2}\right].
\end{split}
\end{align}
Equation \eqref{eq:Monge_finite_diff} is the finite difference analogue of equation \eqref{eq:Vee}. The left-hand side of \eqref{eq:Monge_finite_diff} is a second-order finite difference of $\V$ which we wish to relate to $\Delta\V$, so to this end we introduce the following $\Delta_{\epsilon}$ operator and its properties; see Section \ref{subsec:appendix_Delta_eps} for the proof.

\begin{lemma}
\label{lem:avg_Laplacian}
Fix $\epsilon> 0$ and for a function $f:\R^{\dd}\to\R$ denote
\begin{equation}
\label{eq:avg_Laplacian_def}
\Delta_{\epsilon}f(x):=\avgint_{\partial  B_{\epsilon}(0)}[f(x+y)-f(x)]\diff y,
\end{equation}
where
\[
\avgint_{\partial  B_{\epsilon}(0)}f:=\frac{1}{|\partial  B_{\epsilon}(0)|}\int_{\partial  B_{\epsilon}(0)}f.
\]
Then,
\begin{align}
\label{eq:Delta_eps_lim}
\lim_{\epsilon\to 0}\frac{\Delta_{\epsilon}f(x)}{\epsilon^{2}}=\frac{\Delta f(x)}{2\dd},
\end{align}
where $\Delta f$ is the distributional Laplacian of $f$. Further, if $\Delta f\le \ell$, then 
\begin{align}
\label{eq:Delta_eps_bound}
\Delta_{\epsilon}f\le \frac{\ell}{\dd}\frac{\epsilon^2}{2},\quad \quad\forall~\epsilon>0.
\end{align}
\end{lemma}
We now integrate \eqref{eq:Monge_finite_diff} over $y$ with respect to the uniform measure on the sphere $\partial B_{\epsilon}(x)$ of radius $\epsilon$ centered at $x$. By the assumption $\Delta V\le \lsh\dd$, together with \eqref{eq:Delta_eps_bound}, we have
\begin{align}
\begin{split}
\epsilon^2 \lsh\ge 2\Delta_{\epsilon}V(x)&=2\avgint_{\partial  B_{\epsilon}(0)}\left[\W(\nabla\OT(x+y))-\W(\nabla\OT(x))\right]\diff y\label{eq:Monge_avg}\\
&-\avgint_{\partial  B_{\epsilon}(0)}\log\left[\frac{\det \nabla^2\OT(x+y)\det \nabla^2\OT(x-y)}{(\det \nabla^2\OT(x))^2}\right]\diff y.
\end{split}
\end{align}
On the other hand, by the assumption $\nabla^2\W  \succeq \lc \Id_{\dd}$, 
\begin{equation}
\label{eq:cvx_W}
\W(\nabla\OT(x+y))-\W(\nabla\OT(x))\ge \langle (\nabla \W)(\nabla\OT(x)),\nabla\OT(x+y)-\nabla\OT(x)\rangle+\frac{\lc}{2}|\nabla\OT(x+y)-\nabla\OT(x)|^2,
\end{equation}
which is the analogue of \eqref{eq:Vee_inq}. Combining \eqref{eq:Monge_avg} and \eqref{eq:cvx_W} implies 
\begin{align}
\begin{split}
\epsilon^2 \lsh&\ge \lc\avgint_{\partial  B_{\epsilon}(0)}|\nabla\OT(x+y)-\nabla\OT(x)|^2\diff y\label{eq:Monge_avg_cvx}\\
&+\avgint_{\partial  B_{\epsilon}(0)}\left\{2\langle (\nabla \W)(\nabla\OT(x)),\nabla\OT(x+y)-\nabla\OT(x)\rangle-\log\left[\frac{\det \nabla^2\OT(x+y)\det \nabla^2\OT(x-y)}{(\det \nabla^2\OT(x))^2}\right]\right\}\diff y.
\end{split}
\end{align}
Next we will multiply both sides of \eqref{eq:Monge_avg_cvx} by $(\Delta_{\epsilon}\OT(x))^p$ and integrate against $e^{-V(x)}\diff x$. This requires showing that some integrals are finite. 

\begin{lemma}
\label{lem:integrate_finite}
For every $p>0$ we have
\begin{align}
\begin{split}
&\epsilon^2 \lsh\int (\Delta\OT_{\epsilon}(x))^p e^{-V(x)}\diff x\ge  \lc\int\avgint_{\partial  B_{\epsilon}(0)}\left[|\nabla\OT(x+y)-\nabla\OT(x)|^2 \right]\diff y(\Delta_{\epsilon}\OT(x))^p e^{-V(x)}\diff x\label{eq:2_terms}\\
&+\int (\Delta_{\epsilon}\OT(x))^p e^{-V(x)}\\
&\cdot\left\{\avgint_{\partial  B_{\epsilon}(0)}\left\{ 2\langle \nabla\W(\nabla \OT(x)),\nabla \OT(x+y)-\nabla \OT(x)\rangle-\log\left[\frac{\det \nabla^2\OT(x+y)\det \nabla^2\OT(x-y)}{(\det \nabla^2\OT(x))^2}\right]\right\}\diff y\right\}\diff x,
\end{split}
\end{align}
and all the integrals are finite. 
\end{lemma}
\begin{proof}
Once we show that the integrals are finite, Equation \eqref{eq:2_terms} follows by multiplying both sides of \eqref{eq:Monge_avg_cvx} by $(\Delta_{\epsilon}\OT)^p$ and integrating against $e^{-V}\diff x$. Note that at the moment we are only interested  in these quantities to be finite and not uniformly bounded in \(\epsilon\). The integral $\int (\Delta_{\epsilon}\OT)^p e^{-V}\diff x$ is finite by \eqref{eq:Caffarelli_bound} and \eqref{eq:Delta_eps_bound}. The  integral $\int\avgint_{\partial  B_{\epsilon}(0)}\left[|\nabla\OT(x+y)-\nabla\OT(x)|^2 \right]\diff y(\Delta_{\epsilon}\OT)^p e^{-V}\diff x$ is finite since by the fundamental theorem of calculus we can write the difference $\nabla\OT(x+y)-\nabla\OT(x)$ as an integral along a path between $x+y$ and $x$ of $\nabla^2\OT$ multiplied by a vector of length $\le\epsilon$, so the difference is bounded since $\nabla^2\OT$ is bounded by  \eqref{eq:Caffarelli_bound}. To bound the term $\langle \nabla\W(\nabla \OT(x)),\nabla \OT(x+y)-\nabla \OT(x)\rangle$ use Assumption \ref{assump}(5) according to which $\nabla W$ grows at most linearly, and the fact that again   $\nabla \OT(x+y)-\nabla \OT(x)$ can be bounded by a constant, to see that $\int (\Delta_{\epsilon}\OT)^p\avgint_{\partial  B_{\epsilon}(0)} 2\langle \nabla\W(\nabla \OT(x)),\nabla \OT(x+y)-\nabla \OT(x)\rangle e^{-V}$ is finite since $\target$ has a finite first moment. Finally, to bound the term
  \[
    \int\avgint_{\partial  B_{\epsilon}(0)}\left| \log\left[\frac{\det \nabla^2\OT(x+y)\det \nabla^2\OT(x-y)}{(\det \nabla^2\OT(x))^2}\right]\right| e^{-V(x)}\diff x \diff y 
  \]
  we use identity \eqref{eq:Monge_avg} so we need to bound
  \[
    \int\avgint_{\partial  B_{\epsilon}(0)}|\W(\nabla\OT(x+y))-\W(\nabla\OT(x))|(\Delta_{\epsilon}\OT)^p  e^{-V}\diff y\diff x
  \]
  as well as
  \[
    \int |\Delta_{\epsilon}\V| e^{-V}\diff x.
  \]
  For the first term, again by \eqref{eq:Caffarelli_bound}, the distance between $\nabla\OT(x+y)$ and $\nabla\OT(x)$ is bounded by a constant, so by Assumption \ref{thm:main_OT_Lp}(5), $\avgint_{\partial  B_{\epsilon}(0)}|\W(\nabla\OT(x+y))-\W(\nabla\OT(x))|$ is linear in $\nabla\OT$ (at some point between $x+y$ and $x$) and, since $\nabla^2\Phi$ is bounded,  $\int\avgint_{\partial  B_{\epsilon}(0)}|\W(\nabla\OT(x+y))-\W(\nabla\OT(x))|(\Delta_{\epsilon}\OT)^p  e^{-V}\diff y\diff x$ is finite since $\target$ has a finite first moment. Finally, for the term $ \int |\Delta_{\epsilon}\V |e^{-V}\diff x$,  we use the definition of $\Delta_{\epsilon}$ and note that at infinity $\V$ grows quadratically so  $\int |\Delta_{\epsilon}\V |e^{-V}\diff x$ is finite since $\source$ has a finite second moment.
\end{proof}

We will show in Proposition \ref{prop:noneg_term} below that the second term in \eqref{eq:2_terms} is nonnegative, which is the replacement of \eqref{eq:opt_condition}-\eqref{eq:Tr_noneg}. Let us complete the proof of Theorem \ref{thm:main_OT_Lp} assuming the validity of Proposition \ref{prop:noneg_term}. 
\begin{lemma}
\label{lem:_term_to_thm_proof}
Suppose that for every $p>0$,
\begin{align}
\label{eq:1_term}
&\epsilon^2 \lsh\int (\Delta_{\epsilon}\OT(x))^p e^{-V(x)}\diff x\ge  \lc\int\avgint_{\partial  B_{\epsilon}(0)}\left[|\nabla\OT(y+x)-\nabla\OT(x)|^2 \right]\diff y(\Delta_{\epsilon}\OT(x))^p e^{-V(x)}\diff x.
\end{align}
Then,
\begin{align}
\label{eq:infty_bound}
\|\Delta\OT\|_{L^{\infty}(\source)}\le \dd\sqrt{ \frac{\lsh}{\lc}}.
\end{align}
\end{lemma}
\begin{proof}
Dividing both sides of \eqref{eq:1_term} by $\epsilon^{2p+2}$, and using \eqref{eq:Delta_eps_lim}, we get 
\begin{align*}
 \lsh \int \left(\frac{\Delta \OT}{2\dd}\right)^p e^{-V}\diff x&\ge \lc  \int\left\{\lim_{\epsilon\to 0}\frac{1}{\epsilon^2}\avgint_{\partial  B_{\epsilon}(0)}\left[|\nabla\OT(y+x)-\nabla\OT(x)|^2 \right]\diff y\right\}  \left(\frac{\Delta \OT}{2\dd}\right)^p e^{-V}\diff x\\
&=\lc \int\left\{\avgint_{\partial  B_{1}(0)}\lim_{\epsilon\to 0}\left|\frac{\nabla\OT(x+\epsilon y)-\nabla\OT(x)}{\epsilon}\right|^2\diff y\right\} \left(\frac{\Delta \OT)}{2\dd}\right)^p e^{-V}\diff x\\
&=\lc \int\left\{\avgint_{\partial  B_{1}(0)}\left| \nabla^2\OT(x)y\right|^2\diff y\right\}  \left(\frac{\Delta \OT}{2\dd}\right)^p e^{-V}\diff x.
\end{align*}
Since
\begin{align*}
\avgint_{\partial  B_{1}(0)}\left| \nabla^2\OT(x)y\right|^2\diff y=\frac{\Tr[(\nabla^2\OT(x))^2]}{
\dd}\ge \frac{(\Delta \OT(x))^2}{\dd^2},
\end{align*}
it follows that
\[
 \lsh\int  \left(\frac{\Delta \OT}{2\dd}\right)^p e^{-V}\diff x\ge \lc \int \frac{(\Delta \OT)^2}{\dd^2}  \left(\frac{\Delta \OT}{2\dd}\right)^p e^{-V}\diff x,
\]
and hence,
\begin{align}
\label{eq:p_p+2}
\int (\Delta \OT)^p e^{-V}\ge\frac{1}{\dd^2} \frac{\lc}{\lsh}  \int (\Delta \OT)^{p+2}  e^{-V}\diff x.
\end{align}
By H\^older's inequality, with exponents $\frac{p+2}{2},\frac{p+2}{p}$, we have
\begin{align}
\label{eq:Holder}
\int (\Delta \OT)^p e^{-V}\le \left(\int \left[(\Delta \OT)^p\right]^{\frac{p+2}{p}} e^{-V} \right)^{\frac{p}{p+2}}\left(\int 1^{\frac{p}{p+2}}e^{-V}\right)^{\frac{p+2}{p}}=\left(\int (\Delta\OT)^{p+2}e^{-V}\right)^{\frac{p}{p+2}},
\end{align}
so combining \eqref{eq:Holder} and \eqref{eq:p_p+2} we get 
\[
\left(\int (\Delta\OT)^{p+2}e^{-V}\right)^{\frac{p}{p+2}}\ge  \frac{1}{\dd^2}\frac{\lc}{\lsh}   \int (\Delta \OT)^{p+2}  e^{-V}.
\]
Relabeling $p\mapsto 2p$, the above can be written as 
\begin{align}
\label{eq:q_bound_proof}
\|(\Delta\OT)^2\|_{L^{p+1}(\source)}\le\dd^2 \frac{\lsh}{\lc}.
\end{align}
 Inequality \eqref{eq:infty_bound} follows by taking $p\to\infty$,
\begin{align}
\label{eq:infty_bound_proof}
\|\Delta\OT\|_{L^{\infty}(\source)}\le \dd\sqrt{ \frac{\lsh}{\lc}}.
\end{align}
\end{proof}
It remains to show that the second term in \eqref{eq:2_terms} is nonnegative, which was first observed by Kolesnikov in \cite{MR3201654}. Here we essentially follow his proof with some minor adaptations to our case.
\begin{proposition}
\label{prop:noneg_term}
\begin{align}
\begin{split}
&\int (\Delta_{\epsilon}\OT(x))^p e^{-V(x)}\\\label{eq:noneg_term}
&\cdot\left\{\avgint_{\partial  B_{\epsilon}(0)}\left\{ 2\langle \nabla\W(\nabla \OT(x)),\nabla \OT(x+y)-\nabla \OT(x)\rangle-\log\left[\frac{\det \nabla^2\OT(x+y)\det \nabla^2\OT(x-y)}{(\det \nabla^2\OT(x))^2}\right]\right\}\diff y\right\}\diff x\ge 0.
\end{split}
\end{align}
\end{proposition}
\begin{proof}
Let $\nabla\Psi=(\nabla \OT)^{-1}$ be the optimal transport map between  $\diff \target=e^{-\W}\diff x$ to $\diff \source=e^{-\V}\diff x$. The key estimate required to obtain \eqref{eq:noneg_term} is contained in the following proposition.
\begin{proposition}
\label{prop:nablaW_term}
\begin{align}
\begin{split}
&\int (\Delta_{\epsilon}\OT(x))^p e^{-V(x)}\avgint_{\partial  B_{\epsilon}(0)} 2\langle \nabla\W(\nabla \OT(x)),\nabla \OT(x+y)-\nabla \OT(x)\rangle\diff y\diff x\\ \label{eq:nablaW_term}
&\ge 2\avgint_{\partial  B_{\epsilon}(0)}\left\{\int (\Delta_{\epsilon}\OT(\nabla \Psi(x)))^p  \left\{\Tr\left[\nabla^2\OT(\nabla\Psi(x)+y)\nabla^2\Psi(x)\right]-\dd\right\}\diff \target(x)\right\} \diff y.
\end{split}
\end{align}
\end{proposition}
Assuming the estimate \eqref{eq:nablaW_term} let us prove the estimate \eqref{eq:noneg_term}. Using the symmetry of $\partial  B_{\epsilon}(0)$ under reflection $y\mapsto -y$  we have 
\begin{align*}
&\int (\Delta_{\epsilon}\OT)^p e^{-V}\\
&\cdot\left\{\avgint_{\partial  B_{\epsilon}(0)}\left\{ 2\langle \nabla\W(\nabla \OT),\nabla \OT(x+y)-\nabla \OT(x)\rangle-\log\left[\frac{\det \nabla^2\OT(x+y)\det \nabla^2\OT(x-y)}{(\det \nabla^2\OT(x))^2}\right]\right\}\diff y\right\}\diff x\\
&\overset{\eqref{eq:nablaW_term}}{\ge} 2\avgint_{\partial  B_{\epsilon}(0)}\left\{\int (\Delta_{\epsilon}\OT(\nabla \Psi))^p  \left\{\Tr\left[\nabla^2\OT(\nabla\Psi(x)+y)\nabla^2\Psi(x)\right]-\dd\right\}\diff \target\right\} \diff y\\
&-2\int (\Delta_{\epsilon}\OT)^p \avgint_{\partial  B_{\epsilon}(0)}\log\det\left[\nabla^2\OT(x+y)(\nabla^2\OT(x))^{-1}\right]\diff y\,e^{-V}\diff x\\
&= 2\avgint_{\partial  B_{\epsilon}(0)}\left\{\int  \left\{\Tr\left[\nabla^2\OT(\nabla\Psi(x)+y)\nabla^2\Psi(x)\right]-\dd\right\}(\Delta_{\epsilon}\OT(\nabla \Psi))^p \diff \target\right\} \diff y\\
&-2\avgint_{\partial  B_{\epsilon}(0)}\left\{\int\log\det\left[\nabla^2\OT(\nabla \Psi(x)+y)(\nabla^2\OT(\nabla \Psi(x)))^{-1}\right](\Delta_{\epsilon}\OT(\nabla\Psi))^p\diff \target\right\}\diff y\\
&=2\avgint_{\partial  B_{\epsilon}(0)}\left\{\left[\Tr[A(x)]-\dd-\log \det A(x)\right](\Delta_{\epsilon}\OT(\nabla\Psi))^p\diff \target\right\}\diff y,
\end{align*}
where $A(x)=\nabla^2\OT(\nabla \Psi(x)+y)(\nabla^2\OT(\nabla \Psi(x)))^{-1}$. By \cite[p. 8]{kolesnikov2011mass}, $\Tr[A(x)]-\dd-\log \det A(x)\ge 0$ for every $x$, which completes the proof. 
\end{proof}
It remains to prove  Proposition \ref{prop:nablaW_term}.
\begin{proof}[Proof of Proposition \ref{prop:nablaW_term}]
Let $\nabla\Psi=(\nabla \OT)^{-1}$ be the optimal transport map between  $\diff \target=e^{-\W}\diff x$ to $\diff \source=e^{-\V}\diff x$. The estimate \eqref{eq:nablaW_term} will obtained as a consequence of integration by parts. But in order to justify the vanishing of the boundary terms in the integration by parts we need to work with cutoff functions. Let $\{\eta_k\}$ be a sequence of compactly supported smooth functions $\eta_k:\R^{\dd}\to \R$ such that 
\begin{itemize}
\item $0\le \eta_k\le 1$ for every $k$.
\item $\lim_{k\to\infty}\eta_k=1$ uniformly on every compact set.
\item $\lim_{k\to\infty}\int |\nabla\eta_k(x)|^2 \diff\target(x)=0$.
\end{itemize}
In the computations below we will simplify the notation and write
\[
\findiff_y(x):=\nabla \OT(x+y)-\nabla \OT(x),
\]
omitting the dependence on $x$ when it is clear from context, $\findiff_y:=\findiff_y(x)$.

By the Monge-Amp\`ere equation \eqref{eq:Monge_proof} (and since the integrals can be exchanged by the proof of Lemma \ref{lem:integrate_finite}),  multiplying the integrand in the left-hand side of \eqref{eq:nablaW_term} by $\eta_k(\nabla\OT(x))$, and integrating against $e^{-\V(x)}\diff x$, yields
\begin{align*}
&\int\left\{ (\Delta_{\epsilon}\OT)^p\avgint_{\partial  B_{\epsilon}(0)}2\langle \nabla\W(\nabla \OT),\findiff_y\rangle\diff y\right\}\eta_k(\nabla\OT)\diff \source\\
&=\int\left\{ (\Delta_{\epsilon}\OT)^p\avgint_{\partial  B_{\epsilon}(0)}2\langle \nabla\W(\nabla \OT),\findiff_y\rangle\diff y\right\}\eta_k(\nabla\OT)e^{-\W(\nabla\OT)}\det\nabla^2\OT\diff x\\
&=-\int\left\{ (\Delta_{\epsilon}\OT)^p\avgint_{\partial  B_{\epsilon}(0)}2\langle \nabla(e^{-\W(\nabla\OT)}),(\nabla^2\OT)^{-1}\findiff_y\rangle\diff y\right\}\eta_k(\nabla\OT)\det\nabla^2\OT\diff x\\
&=-\avgint_{\partial  B_{\epsilon}(0)}\left\{\int  2\langle \nabla(e^{-\W(\nabla\OT)}),\Cof(\nabla^2\OT)\findiff_y\rangle(\Delta_{\epsilon}\OT)^p\eta_k(\nabla\OT)\diff x\right\}\diff y.
\end{align*}
By integration by parts in the $x$ variable (which has no boundary terms because $\eta_k$ is compactly supported), and using the fact that the cofactor matrix is divergence free, we get
\begin{align*}
  -\int&  \langle \nabla(e^{-\W(\nabla\OT)}),\Cof(\nabla^2\OT)\findiff_y\rangle(\Delta_{\epsilon}\OT)^p\eta_k(\nabla\OT)
  \\
  &=\int e^{-\W(\nabla\OT)} \Div\left[\Cof(\nabla^2\OT)\findiff_y(\Delta_{\epsilon}\OT)^p\eta_k(\nabla\OT)\right]
  \\
  &=\int e^{-\W(\nabla\OT)} \Tr\left[\Cof(\nabla^2\OT)\nabla[\findiff_y(\Delta_{\epsilon}\OT)^p\eta_k(\nabla\OT)]\right]
  \\
  &=\int e^{-\W(\nabla\OT)} \Tr\left[\Cof(\nabla^2\OT)\nabla[\findiff_y(\Delta_{\epsilon}\OT)^p]\right]\eta_k(\nabla\OT)
  \\
&\quad +\int e^{-\W(\nabla\OT)} \Tr\left[\Cof(\nabla^2\OT)\left\{\left((\Delta_{\epsilon}\OT)^p\findiff_y\right)\otimes\nabla[\eta_k(\nabla\OT)]\right\}\right],
\end{align*}
where we use the convention
$(\nabla u(x))_{ij}=\partial_iu_j(x)$ for a vector field $u$, and $(w\otimes v)_{ij}=w_jv_i$ for vectors $v,w$. Hence, our goal is to lower bound the term
\begin{align}
\begin{split}
&\int\left\{ (\Delta_{\epsilon}\OT)^p\avgint_{\partial  B_{\epsilon}(0)}2\langle \nabla\W(\nabla \OT),\findiff_y\rangle\diff y\right\}\eta_k(\nabla\OT)\diff \source\\
&=2\avgint_{\partial  B_{\epsilon}(0)}\int e^{-\W(\nabla\OT)} \Tr\left[\Cof(\nabla^2\OT)\nabla[\findiff_y(\Delta_{\epsilon}\OT)^p]\right]\eta_k(\nabla\OT)\diff x\diff y\\\label{eq:2terms}
&+2\avgint_{\partial  B_{\epsilon}(0)}\int e^{-\W(\nabla\OT)} \Tr\left[\Cof(\nabla^2\OT)\left\{\left((\Delta_{\epsilon}\OT)^p\findiff_y\right)\otimes\nabla[\eta_k(\nabla\OT)]\right\}\right]\diff x\diff y.
\end{split}
\end{align}
The second term on the right-hand side of \eqref{eq:2terms} will be shown to vanish as $k\to\infty$ so it suffices to lower bound the first term on the right-hand side of \eqref{eq:2terms}, and then take the limit $k\to \infty$.  We start with the first term in \eqref{eq:2terms}. 

\begin{lemma}
\label{lem:1term}
\begin{align}
\begin{split}
&\avgint_{\partial  B_{\epsilon}(0)}\int e^{-\W(\nabla\OT(x))} \Tr\left[\Cof(\nabla^2\OT(x))\nabla[\left(\nabla \OT(x+y)-\nabla \OT(x)\right)(\Delta_{\epsilon}\OT(x))^p]\right]\eta_k(\nabla\OT(x))\diff x\diff y\\\label{eq:1term}
&\ge \avgint_{\partial  B_{\epsilon}(0)}\int\left \{\Tr\left[\nabla^2\OT(\nabla\Psi(x)+y)\nabla^2\Psi(x)\right]-\dd\right\}(\Delta_{\epsilon}\OT(\nabla\Psi(x)))^p \eta_k(x)\diff \target(x)\diff y.
\end{split}
\end{align}
\end{lemma}
\begin{proof}
Let $\nabla\Psi=(\nabla \OT)^{-1}$ be the optimal transport map between  $\diff \target=e^{-\W}\diff x$ to $\diff \source=e^{-\V}\diff x$, and recall $\nabla^2\Psi(x)=(\nabla^2\OT)^{-1}(\nabla\Psi(x))$. Using the chain rule we have, for any vector field $u$ in $\R^{\dd}$,  $\Div_x[u(\nabla \Psi(x))]=\Tr[\nabla^2 \Psi(x)\nabla u(\nabla \Psi(x))]$, so with
\[
\findiff_y:=\nabla \OT(x+y)-\nabla \OT(x), \quad  \findiff_y(\nabla\Psi):=\findiff_y(\nabla\Psi(x))=\nabla \OT(\nabla\Psi(x)+y)-\nabla \OT(\nabla\Psi(x)),
\]
we have
\begin{align*}
&\int e^{-\W(\nabla\OT)} \Tr\left[\Cof(\nabla^2\OT)\nabla[(\Delta_{\epsilon}\OT)^p\findiff_y]\right]\eta_k(\nabla\OT)\diff x=\int \Tr\left[\nabla^2\Psi(\nabla\OT)\nabla[(\Delta_{\epsilon}\OT)^p\findiff_y]\right]\eta_k(\nabla\OT)\diff \source\\
&=\int \Tr\left[\nabla^2\Psi\nabla[\findiff(\nabla\Psi)(\Delta_{\epsilon}\OT(\nabla\Psi))^p]\right]\eta_k\diff \target=\int \Div[(\Delta_{\epsilon}\OT(\nabla\Psi))^p \findiff_y(\nabla\Psi)]\eta_k\diff \target\\
&=\int \Div[\left(\nabla \OT(\nabla\Psi(x)+y)-x\right)](\Delta_{\epsilon}\OT(\nabla\Psi))^p \eta_k\diff \target+\int \langle\findiff_y(\nabla\Psi),\nabla(\Delta_{\epsilon}\OT(\nabla\Psi))^p \rangle\eta_k\diff \target.
\end{align*}
Hence, 
\begin{align}
\begin{split}
&\avgint_{\partial  B_{\epsilon}(0)}\int e^{-\W(\nabla\OT)} \Tr\left[\Cof(\nabla^2\OT)\nabla[(\Delta_{\epsilon}\OT)^p\findiff_y]\right]\eta_k(\nabla\OT)\diff x\diff y\\
&=\avgint_{\partial  B_{\epsilon}(0)}\int \Div[\left(\nabla \OT(\nabla\Psi(x)+y)-x\right)](\Delta_{\epsilon}\OT(\nabla\Psi))^p \eta_k\diff \target\diff y+\avgint_{\partial  B_{\epsilon}(0)}\int \langle\findiff_y(\nabla\Psi),\nabla(\Delta_{\epsilon}\OT(\nabla\Psi))^p \rangle\eta_k\diff \target\diff y.\label{eq:1term_equiv}
\end{split}
\end{align}
 
Let us bound the two terms in \eqref{eq:1term_equiv} separately. For the first term  in \eqref{eq:1term_equiv}  we use the chain rule to get
\begin{align*}
&\Div[\nabla \OT(\nabla\Psi(x)+y)-x]=\Div[\nabla \OT(\nabla\Psi(x)+y)]-\dd\ge \Tr\left[\nabla^2\OT(\nabla\Psi(x)+y)\nabla^2\Psi(x)\right]-\dd.
\end{align*}
Using $\nabla^2\Psi(x)=(\nabla^2\OT)^{-1}(\nabla\Psi(x))$ gives
\begin{align}
\begin{split}
&\avgint_{\partial  B_{\epsilon}(0)}\int \Div[\left(\nabla \OT(\nabla\Psi(x)+y)-x\right)](\Delta_{\epsilon}\OT(\nabla\Psi))^p \eta_k\diff \target\diff y\\\label{eq:1term_1}
&\ge \avgint_{\partial  B_{\epsilon}(0)}\int\left \{\Tr\left[\nabla^2\OT(\nabla\Psi(x)+y)\nabla^2\Psi(x)\right]-\dd\right\}(\Delta_{\epsilon}\OT(\nabla\Psi))^p \eta_k\diff \target\diff y.
\end{split}
\end{align} 
For the second term  in \eqref{eq:1term_equiv}  we use the definition of $\Delta_{\epsilon}$ to write
\begin{align}
\begin{split}
&\avgint_{\partial  B_{\epsilon}(0)}\int \left\langle\nabla \findiff_y(\nabla\Psi), \nabla\left((\Delta_{\epsilon}\OT(\nabla \Psi))^p\right) \right\rangle \eta_k\diff \target\diff y
\\ \label{eq:1term_2}
&=p\avgint_{\partial  B_{\epsilon}(0)}\int (\Delta_{\epsilon}\OT(\nabla \Psi))^{p-1} \left\langle \findiff_y(\nabla\Psi), \nabla^2\Psi \nabla\Delta_{\epsilon}\OT(\nabla \Psi)\right\rangle \eta_k\diff \target\diff y\\
&=p\int (\Delta_{\epsilon}\OT(\nabla \Psi))^{p-1} \left\langle \avgint_{\partial  B_{\epsilon}(0)}\findiff_y(\nabla\Psi)\diff y,\nabla^2\Psi \nabla\Delta_{\epsilon}\OT(\nabla \Psi)\right\rangle \eta_k\diff \target \\
&\!\!\!\overset{\eqref{eq:avg_Laplacian_def}}{=} p\int (\Delta_{\epsilon}\OT(\nabla \Psi))^{p-1} \left\langle (\Delta_{\epsilon}\nabla \OT)(\nabla\Psi),\nabla^2\Psi \nabla\Delta_{\epsilon}\OT(\nabla \Psi)\right\rangle  \eta_k\diff \target\\
&=p\int (\Delta_{\epsilon}\OT(\nabla \Psi))^{p-1} \left\langle \nabla\Delta_{\epsilon}\OT(\nabla \Psi),\nabla^2\Psi \nabla\Delta_{\epsilon}\OT(\nabla \Psi)\right\rangle  \eta_k\diff \target\ge 0,
\end{split}
\end{align}
where we used the convexity of $\OT$ and $\Psi$ as well as $\eta_k\ge 0$. Combining \eqref{eq:1term_1} and \eqref{eq:1term_2} yields \eqref{eq:1term}.
\end{proof}
We now turn to the second term in \eqref{eq:2terms}. 
\begin{lemma}
\label{lem:2term}
\begin{align}
\begin{split}
&2\int e^{-\W(\nabla\OT(x))}  \Tr\left[\Cof(\nabla^2\OT(x))\left\{\left(\left(\nabla \OT(x+y)-\nabla \OT(x)\right)(\Delta_{\epsilon}\OT(x))^p\right)\otimes\nabla[\eta_k(\nabla\OT(x))]\right\}\right]\diff x\\\label{eq:2term}
&\le 2\left(\int |\left(\nabla \OT(\nabla\Psi(x)+y)-\nabla \OT(\nabla\Psi(x))\right)(\Delta_{\epsilon}\OT(\nabla\Psi(x)))^p|^2\diff \target(x)\right)^{\frac{1}{2}}\left(\int |\nabla\eta_k(x)|^2\diff \target(x)\right)^{\frac{1}{2}}.
\end{split}
\end{align}
\end{lemma}
\begin{proof}
Given a symmetric matrix $M$ and vectors $u,v$ we have $\Tr[M^{-1}\{u\otimes (Mv)\}]=\langle u,v\rangle$, so applying this identity we get, with $\findiff_y:=\nabla \OT(x+y)-\nabla \OT(x)$,
\begin{align*}
&\int e^{-\W(\nabla\OT)}  \Tr\left[\Cof(\nabla^2\OT)\left\{\left(\findiff_y(\Delta_{\epsilon}\OT)^p\right)\otimes\nabla[\eta_k(\nabla\OT)]\right\}\right]\diff x\\
&=\int \Tr\left[\left\{\left[\findiff(\Delta_{\epsilon}\OT)^p\right]\otimes [(\nabla\eta_k)(\nabla\OT)]\right\}\right]\diff \source=\int \Tr\left[\left\{\left[\findiff_y(\nabla\Psi)(\Delta_{\epsilon}\OT(\nabla\Psi))^p\right]\right\}\otimes\nabla\eta_k\right]\diff \target\\
&\le \left(\int |\findiff_y(\nabla\Psi)(\Delta_{\epsilon}\OT(\nabla\Psi))^p|^2\diff \target\right)^{\frac{1}{2}}\left(\int |\nabla\eta_k|^2\diff \target\right)^{\frac{1}{2}}.
\end{align*}
\end{proof}
Let us now take the limits $k\to\infty$ of the two terms in \eqref{eq:2terms}. For the first term in \eqref{eq:2terms} we first use \eqref{eq:1term} and then take the limit $k\to\infty$ in \eqref{eq:1term}. We can move the limit past the integrals in \eqref{eq:1term} by dominated convergence theorem using $\eta_k\le 1$, and the fact that the rest of the integrand is integrable (as in  the proof of Lemma \ref{lem:integrate_finite}). Thus, if the second term in \eqref{eq:2terms} vanishes in the limit $k\to\infty$ we will get \eqref{eq:nablaW_term}. To show that  the second term in \eqref{eq:2terms} vanishes in the limit $k\to\infty$ it suffices to show that the first term on the right-hand side of \eqref{eq:2term} is finite since by assumption $\lim_{k\to\infty}\int |\nabla\eta_k(x)|^2 \diff\target(x)=0$. The first term on the right-hand side of \eqref{eq:2term} is finite by the same argument as in the proof of Lemma \ref{lem:integrate_finite}. 
\end{proof}
To summarize, the combination of Lemma \ref{lem:1term} and  Lemma \ref{lem:2term} yields Proposition \ref{prop:nablaW_term}, which in turn implies Proposition \ref{prop:noneg_term}. It follows that  the second term in \eqref{eq:2_terms} is nonnegative and the proof of Theorem \ref{thm:main_OT_Lp} is complete by Lemma \ref{lem:_term_to_thm_proof}. 
\end{proof}

\begin{remark}
\label{rem:modification}
Our proof of Theorem \ref{thm:main_OT_Lp} follows the proof of \cite[Theorem 6.2]{kolesnikov2011mass}. However, one important modification we need to make is the introduction of the $\Delta_{\epsilon}$ operator. In \cite{kolesnikov2011mass}, the lack of regularity of $\V$ and $\nabla\OT$ is remedied by considering finite differences (rather than actual derivatives), $\V(x+e)-\V(x)$ and $\nabla\OT(x+e)-\nabla\OT(x)$, for vectors $e\in\R^{\dd}$. The finite difference $\V(x+e)-\V(x)$ is then controlled by explicit assumptions on the second directional derivatives of $\V$, which in turn leads to control on the second directional derivatives of $\OT$. In contrast, we only have at our disposal the control of $\Delta \V$, so the finite differences approach just outlined does not work. The operator $\Delta_{\epsilon}$ is the appropriate replacement to the finite differences scheme. 
\end{remark}

\section{Majorization}
\label{sec:major}
Theorem \ref{thm:volume-contraction_OT_intro} and Theorem \ref{thm:majorization_intro} were established in Section \ref{subsec:vol_contract_majorization}, so with regard to majorization it remains to prove Theorem \ref{thm:mono_intro} and Theorem \ref{thm:entropy_stability_intro}. We begin with Theorem \ref{thm:mono_intro}.

\begin{theorem}[Monotonicity along Wasserstein geodesics]
\label{thm:mono}
Let $\diff\source=e^{-\V}\diff x$ and $\diff\target=e^{-\W}\diff x$ be probability measures on $\R^{\dd}$, with $\source$ supported on all of $\R^{\dd}$,  such that there exist $\lsh>0,\lc>0$ with
\[
\Delta \V \le \lsh\dd \quad\text{and}\quad \nabla^2\W \succeq \lc \Id_{\dd}.
\]
Let $\nabla\OT:\R^{\dd}\to\R^{\dd}$ be the Brenier map transporting $\source$ to $\target$, and let $(\prob_t)_{t\in [0,1]}$ be the geodesic in Wasserstein space connecting $\source$ and $\target$, $\prob_t:=(\nabla\OT_t)_{\sharp}\source$ where $\nabla\OT_t(x):=(1-t)x+t\nabla\OT(x)$. Then, if $\frac{\lsh}{\lc}\le 1$,
\begin{equation}
\label{eq:mono_geo}
[0,1]\ni t\mapsto \int_{\R^{\dd}}\cxfun(\prob_t(x))\diff x\textnormal{ is monotonically non-decreasing}
\end{equation}
for every convex function $\cxfun:\R_{\ge 0}\to \R$.
\end{theorem}

\begin{proof}
  Let $\nabla\Psi:=(\nabla\OT)^{-1}$ be Brenier map between $\target$ and $\source$, and let $\nabla\Psi_t(x):=(1-t)\nabla\Psi(x)+tx$ be the Brenier map between $\target=\prob_1$ to $\prob_t$. Let us show that the Brenier map $(\nabla\Psi_t)^{-1}(x)=[(1-t)\nabla\Psi(x)+tx]^{-1}$ between $\prob_t$ and $\prob_1=\target$ satisfies $\Delta[(\nabla\Psi_t)^{-1}]\le \dd$. Indeed, the eigenvalues of $\nabla[(\nabla\Psi_t)^{-1}(x)]$ are
  \[
    \left\{\frac{\lambda_i(x)}{1+t(\lambda_i(x)-1)}\right\}_{i=1}^{\dd},
  \]
  where $\{\lambda_i(x)\}_{i=1}^{\dd}$ are the eigenvalues of $\nabla^2\OT(x)$. Given a fixed $x\in \R^{\dd}$ define the function $\theta:[0,1]\to\R$  by $\theta(t):=\sum_{i=1}^{\dd}\frac{\lambda_i(x)}{1+t(\lambda_i(x)-1)}$. The function $\theta$ is convex and satisfies $\theta(0)\le \dd$, by \eqref{eq:infty_bound_main}, and $\theta(1)=\dd$. It follows that $\theta(t)\le \dd$ for all $t\in [0,1]$, i.e., $\Delta[(\nabla\Psi_t)^{-1}]\le \dd$. In particular, using $\frac{\lsh}{\lc}\le 1$,  $\det[(\nabla\Psi_t)^{-1}]\le 1$ so $\prob_1=\target$ majorizes $\prob_t$.

Moreover, setting \(\nabla \Phi_s(x):=(1-s)x+s \nabla \Phi(x)\) to be the Brenier map between \(\rho_{0}\) and \(\rho_{s}\) we clearly have that \(\Delta \Phi_{s}\le \dd\) and so \(\prob_0=\source\) is majorized by  \(\prob_s\). To show that for \(0<r\le s\) the measure $\prob_{r}$ is majorized by \(\prob_s\), we can apply the first step with \(\prob_{1}\) replaced by \(\prob_{s}\) and \(\Phi\) replaced by \(\Phi_{s}\). Notice indeed that the only property we have used of \(\Phi\) is that \(\Delta \Phi \le \dd\), which is true for \(\Phi_{s}\) as well.
\end{proof}
Let us now prove Theorem \ref{thm:entropy_stability_intro}. 
\begin{theorem}
\label{thm:stab}
Let $\diff\source=e^{-\V}\diff x$ and $\diff\target=e^{-\W}\diff x$ be probability measures on $\R^{\dd}$, with $\source$ supported on all of $\R^{\dd}$,  such that there exist $\lsh>0,\lc>0$ with 
\[
\Delta \V \le \lsh\dd \quad\text{and}\quad \nabla^2\W \succeq \lc \Id_{\dd}.
\]
Let $\nabla\OT:\R^{\dd}\to\R^{\dd}$ be the Brenier map transporting $\source$ to $\target$.  If $\frac{\lsh}{\lc}\le  1$, then
\begin{equation}
\label{eq:stab_proof}
\Ent(\target)- \Ent(\source) \ge   \frac{1}{2\dd^2}\int_{\R^{\dd}}\|\nabla^2\OT-\Id_{\dd}\|_{\text{F}}^2\diff \source,
\end{equation}
where $\|\cdot\|_{\text{F}}$ is the Frobenius norm. In particular,
\begin{equation}
\label{eq:entropy_eq_proof}
\Ent(\target)= \Ent(\source)\quad\quad\Longrightarrow \quad\quad \textnormal{$\target$ is a translate of $\source$.}
\end{equation}
\end{theorem}
\begin{proof}
By \eqref{eq:infty_bound_main}, and the assumption  $\frac{\lsh}{\lc}\le  1$,
\[
\|\Delta\OT\|_{L^{\infty}(\diff x)}\le \dd\sqrt{ \frac{\lsh}{\lc}}\le \dd.
\]
Abusing notation we identify between the measures $\source,\target$ and their densities, so the Monge-Amp\`ere equation reads
\[
\source=\target( \nabla\OT)\det\nabla^2\OT.
\]
By  the change of variables formula,
\begin{align*}
\Ent(\source) &=\int \left(\target(\nabla\OT)\det\nabla^2\OT\right)\log \target( \nabla\OT)+\int \left(\target( \nabla\OT)\det\nabla^2\OT\right)\log \det\nabla^2\OT\nonumber\\
&=\Ent(\target)+\int \log \det\left(\nabla^2\OT \circ(\nabla\OT)^{-1}\right)\diff \target=\Ent(\target)+\int \log \det\left(\nabla^2\OT\right)\diff \source,
\end{align*}
so
\begin{align*}
\Ent(\target)-\Ent(\source) =\sum_{k=1}^{\dd}\int [-\log \lambda_k(x)]\diff \source(x),
\end{align*}
where $0\le \lambda_1(x)\le\cdots\le \lambda_{\dd}(x)$ are the  eigenvalues of $\nabla^2\OT(x)$. Using
\begin{align*}
-\log s\ge -\log t+\frac{t-s}{t}+\frac{(s-t)^2}{2\max\{s,t\}^2},\quad s,t\in (0,\infty),
\end{align*}
 (see \cite[Proof of Lemma 2.5]{MR2672283}),
with $s=\lambda_k(x)$ and $t=1$, we get
\begin{align*}
\Ent(\target)-\Ent(\source) &\ge \sum_{k=1}^{\dd}\int \left\{(1-\lambda_k(x))+\frac{1}{2}\frac{(\lambda_k-1)^2}{\max\{\lambda_{\dd},1\}^2}\right\}\diff \source(x)\\
&\ge \int \left[\dd-\Delta\OT \right]\diff \source+ \frac{1}{2\dd^2}\sum_{k=1}^{\dd}\int(\lambda_k(x)-1)^2\diff \source(x)\\
&\ge  \frac{1}{2\dd^2}\int\|\nabla^2\OT(x) -\Id_{\dd}\|_{\text{F}}^2\diff \source(x),
\end{align*}
where we used $\Delta \OT\le \dd\Rightarrow \lambda_{\dd}\le \dd$ in the second inequality, and $\Delta \OT\le \dd$ in the last inequality. This establishes \eqref{eq:stab_proof}. Equation \eqref{eq:entropy_eq_proof} follows from  \eqref{eq:stab_proof} since
\[
\Ent(\target)= \Ent(\source)\quad\Longrightarrow \quad\nabla^2\OT(x)=\Id\quad \textnormal{for almost-everywhere $x$},
\]
which implies that $\nabla\OT(x)=x+v$ for some $v\in \R^{\dd}$.  
\end{proof}

\section{The Kim-Milman heat flow transport map}
\label{sec:KM}
When the target measure $\target$ is the standard Gaussian measure $\Gaussian$ we can construct  a volume-contracting map based on the heat flow map of Kim and E. Milman  \cite{MR2983070} (following Otto and Villani \cite{MR1760620}). While this transport map exists outside the Gaussian setting \cite{MR2983070,MR4707029}, our techniques to establish volume contraction are restricted to Gaussian targets. In addition, in contrast to the Brenier map, we can only establish volume contraction rather than control of the divergence of the map. Finally, to establish the existence of the heat flow map we require further regularity than those for the optimal transport map. For these reasons, we will keep this section brief and assume whatever regularity is needed. 

We recall the definition of the Ornstein-Uhlenbeck semigroup  $(\Pheat_t)$,
\begin{equation}
\label{eq:OU_def}
\Pheat_t\density(x):=\int_{\R^{\dd}} \density(e^{-t}x+\sqrt{1-e^{-2t}}y)\diff \Gaussian(y),\quad  t\ge 0,\quad x\in \R^{\dd}, \quad \density \in L^1(\Gaussian).
\end{equation}
The construction of the heat flow map between a source measure $\source$ and a Gaussian target $\Gaussian$ is based on the following flow. 
\begin{proposition}
\label{prop:KM_construction}
Let $\prob$ be an absolutely continuous probability measure on $\R^{\dd}$, and let  $\density:=\frac{\diff\prob}{\diff\Gaussian}$. Suppose the ordinary differential equation
\begin{equation}
\label{eq:KM_construction}
\begin{cases}
\partial_t\KM_t(x)=-\nabla\log\Pheat_t\density(\KM_t(x)),\quad \text{for} ~ t>0 \text{ and } x\in \R^{\dd}, \\
\KM_0(x)=x,\quad \text{for}~ x\in \R^{\dd},
\end{cases}
\end{equation}
has a unique smooth solution. Then, the probability measures $\prob_t:=(\KM_t)_{\sharp}\prob$ satisfy
\begin{equation}
\label{eq:heat_flow_path}
\diff\prob_t(x)=\Pheat_t \density(x)\diff\Gaussian(x),\quad\text{for } t\ge 0 \text{ and } x\in \R^{\dd}.
\end{equation} 
\end{proposition}
The identity \eqref{eq:heat_flow_path} follows from the standard switch between the Eulerian and Lagrangian perspectives, as well as the partial differential equation satisfied by $(t,x)\mapsto \Pheat_t \density(x)$ (cf. \cite[\S 1.2]{MR2983070}). The key point is that as $t\to\infty$, $\Pheat_t \density\to 1$, so $\prob_t\to \Gaussian$. In particular, whenever $\KM:=\lim_{t\to\infty}\KM_t$ exists, we find that $\KM$ transports $\prob$ to $\Gaussian$. The map $\KM$ is the \emph{Kim-Milman heat flow map}. The main result of this section is that, when the target measure is Gaussian, the Kim-Milman heat flow map also achieves the bound \eqref{eq:det_bound_main_intro}.

\begin{theorem}
\label{thm:volume_contraction_KM}
Let $\diff\source=e^{-\V}\diff x$ be a probability measure on $\R^{\dd}$ such that there exists $\lsh>0$ with 
\[
\Delta \V(x)\le \lsh\dd \quad \textnormal{for every $x\in\R^{\dd}$}.
\]
Suppose that the flow \eqref{eq:KM_construction}, with $\density:=\frac{\diff \source}{\diff \Gaussian}$, converges to a differentiable limit $\KM:=\lim_{t\to\infty}\KM_t$, and that  $\lim_{t\to\infty}\det\nabla \KM_t=\det\nabla \KM$.  Then, the Kim-Milman heat flow map between $\source$ and the standard Gaussian $\Gaussian$ on $\R^{\dd}$ satisfy
\begin{equation}
\label{eq:volume_contraction_KM}
 \|\det\nabla \KM\|_{L^{\infty}(\diff x)}\le \lsh^{\frac{\dd}{2}}.
 \end{equation}
\end{theorem}
The observation behind the proof of Theorem \ref{thm:volume_contraction_KM} is that $\det \nabla\KM_t$ can be controlled if $\Delta\log \Pheat_tf$ can be controlled. On other hand, by Proposition \ref{prop:OU_smooth}(ii), $\Delta\log \Pheat_tf$ can be lower bounded provided that $\Delta\log f$ can be lower bounded. The next result implements these observations. 

\begin{proposition}
\label{prop:KM_transport_contraction_t}
Fix $\conv \ge 1$ and let $\diff\prob=\density \diff \Gaussian$ be a $(-\conv\dd)$-log-subharmonic measure. Fix $t\ge 0$ and let $\KM_t$ be the  heat flow map from $\prob$ to $(\Pheat_t\density)\Gaussian$. Then, for all $x\in \R^{\dd}$ and $t\ge 0$,
\[
|\det\nabla\KM_t(x)|\le\left[(1-e^{-2t})(\conv-1)+1\right]^{\frac{\dd}{2}}, \quad\textnormal{for every }x\in \R^{\dd}. 
\]
\end{proposition}

\begin{proof}
We start by recalling the Jacobi formula. Let $(M_t)_{t\ge 0}$ be a family of invertible $\dd\times\dd$ matrices such that the map $t\mapsto \det M_t$ is differentiable. Then, 
\begin{equation}
\label{eq:Jacobi_proof}
\partial_t\det M_t=\Tr\left[M_t^{-1}\partial_tM_t\right]\det M_t.
\end{equation}
In order to apply \eqref{eq:Jacobi_proof} we first need to derive the evolution equation for $(\nabla\KM_t)$. The evolution of $(\KM_t)$ is determined by the equation
\begin{equation}
\label{eq:KM_evolution_proof}
\partial_t\KM_t(x)=-\nabla\log  \Pheat_t\density(\KM_t(x)),\quad \KM_0(x)=x,\quad \forall~x\in \R^{\dd},
\end{equation}
so differentiating \eqref{eq:KM_evolution_proof} with respect to $x$ yields the evolution equation of $(\nabla\KM_t)$,
\begin{equation}
\label{eq:KM_der_evolution_proof}
\partial_t\nabla\KM_t(x)=-\nabla^2\log  \Pheat_t\density(\KM_t(x))\nabla\KM_t(x),\quad \nabla\KM_0=\Id,\quad \forall~x\in \R^{\dd}.
\end{equation}
Hence, by the Jacobi formula \eqref{eq:Jacobi_proof} and the cyclic property of the trace,
\begin{align*}
\partial_t\det\nabla\KM_t(x)&= \Tr\left[(\nabla\KM_t(x))^{-1}\partial_t\nabla\KM_t(x)\right]\det\nabla\KM_t(x)\\
&= \Tr\left[(\nabla\KM_t(x))^{-1}\left\{-\nabla^2\log  \Pheat_t\density(\KM_t(x))\nabla\KM_t(x)\right\}\right]\det\nabla\KM_t(x)\\
&=[-\Delta\log  \Pheat_t\density(\KM_t(x))]\det\nabla\KM_t(x).
\end{align*}
We conclude that
\begin{equation}
\label{eq:KM_Jacobi_proof}
\partial_t\det\nabla\KM_t(x)= [-\Delta\log  \Pheat_t\density(\KM_t(x))]\det\nabla\KM_t(x),\quad \det\nabla\KM_0=1,\quad \forall~x\in \R^{\dd}.
\end{equation}
Since $\prob$ is a $(-\conv\dd)$-log-subharmonic measure it follows that $\density$ is a  $-(\conv-1)\dd$-log-subharmonic function. Hence, by Proposition \ref{prop:OU_smooth}(iii),
\[
-\Delta\log  \Pheat_t\density(\KM_t(x))\le \frac{\dd(\conv-1) e^{-2t}}{1+(1-e^{-2t})(\conv-1)},
\]
and it follows from Gr\"onwall's inequality that 
\[
\det\nabla\KM_t(x)\le e^{\int_0^t\frac{\dd(\conv-1) e^{-2s}}{1+(1-e^{-2s})(\conv-1)}\diff s}.
\]
To conclude the proof note that, for every $x\in \R^{\dd}$, $\det \nabla\KM_0(x)=1$ and $\det\nabla\KM_t(x)\neq 0$ for every $t\ge 0$ since $\KM_t$ is a diffeomorphism. Hence, for every $x\in \R^{\dd}$ and $t\ge 0$, $\det\nabla\KM_t(x)>0$, so 
\[
|\det\nabla\KM_t(x)|=\det\nabla\KM_t(x)\le e^{\int_0^t\frac{\dd(\conv-1) e^{-2s}}{1+(1-e^{-2s})(\conv-1)}\diff s}.
\]
The indefinite integral of the integrand inside the exponential is $\log[1+(\conv -1)(1-e^{-2t})]^{\dd/2}$, which completes the proof. 
\end{proof}
\begin{proof}[Proof of Theorem \ref{thm:volume_contraction_KM}]
By Proposition \ref{prop:KM_transport_contraction_t} and the assumption on $\V$,
\begin{equation}
\label{eq:det_Ft_est}
|\det\nabla\KM_t(x)|\le\left[(1-e^{-2t})(\lsh-1)+1\right]^{\frac{\dd}{2}}.
\end{equation}
Taking $t\to\infty$ yields 
\[
|\det\nabla\KM(x)|\le\lsh^{\frac{\dd}{2}}, \quad \text{for all }x\in \R^{\dd}.
\]
\end{proof}

\section{Appendix}
\label{sec:appendix}

\subsection{The  operator $\Delta_{\epsilon}$} 
\label{subsec:appendix_Delta_eps}
In this section we prove some of the properties of the  operator $\Delta_{\epsilon}$. In particular, let us prove Lemma \ref{lem:avg_Laplacian}.
\begin{lemma}
\label{lem:avg_Laplacian_appendix}
Fix $\epsilon> 0$ and for a function $f:\R^{\dd}\to\R$ denote
\[
\Delta_{\epsilon}f(x):=\avgint_{\partial  B_{\epsilon}(0)}[f(x+y)-f(x)]\diff y,
\]
where
\[
\avgint_{\partial  B_{\epsilon}(0)}f:=\frac{1}{|\partial  B_{\epsilon}(0)|}\int_{\partial  B_{\epsilon}(0)}f.
\]
Then,
\begin{align}
\label{eq:Delta_eps_lim_appendix}
\lim_{\epsilon\to 0}\frac{\Delta_{\epsilon}f(x)}{\epsilon^{2}}=\frac{\Delta f(x)}{2\dd},
\end{align}
where $\Delta f$ is the distributional Laplacian of $f$. Further, if $\Delta f\le \ell$, then 
\begin{align}
\label{eq:Delta_eps_bound_appendix}
\Delta_{\epsilon}f\le \frac{\ell}{\dd}\frac{\epsilon^2}{2},\quad \quad\forall~\epsilon>0.
\end{align}
\end{lemma}
\begin{proof}
To prove \eqref{eq:Delta_eps_lim_appendix} first note that letting $ \nabla^2f$ be the distributional derivative of $f$ we have
\begin{align*}
\lim_{\epsilon\to 0}\frac{f(x+\epsilon y)+f(x-\epsilon y)-2f(x)}{\epsilon^2}=\langle \nabla^2f(x)y,y\rangle,
\end{align*}
in the sense of distributions, where the distribution $\langle \nabla^{2} f(x)y, y \rangle$ is defined as
  \[
\left \langle \langle \nabla^{2} f(x)y, y \rangle, \eta \right \rangle =\int f(x) \langle \nabla^{2} \eta( x) y, y \rangle \diff x,
  \]
for test functions $\eta$. Hence, using the symmetry of $\partial B_{\epsilon}(0)$ under reflection, and changing variables $y\mapsto \epsilon y$,
\begin{align*}
\lim_{\epsilon\to 0}\frac{\Delta_{\epsilon}f(x)}{\epsilon^2}&=\lim_{\epsilon\to 0}\frac{1}{2\epsilon^2}\left\{\avgint_{\partial B_{\epsilon}(0)}2[f(x+y)-f(x)]\diff y\right\}\\
&=\frac{1}{2}\lim_{\epsilon\to 0}\left\{\avgint_{\partial B_{\epsilon}(0)}\frac{f(x+y)+f(x-y)-2f(x)}{\epsilon^2}\diff y\right\}\\
&=\frac{1}{2}\lim_{\epsilon\to 0}\left\{\frac{1}{\epsilon^{\dd-1}}\avgint_{\partial B_{1}(0)}\frac{f(x+\epsilon y)+f(x-\epsilon y)-2f(x)}{\epsilon^2}\epsilon^{\dd-1}\diff y\right\}\\
&=\frac{1}{2}\avgint_{\partial B_{1}(0)}\langle \nabla^2f(x)y,y\rangle \diff y.
\end{align*}
For each $i,j\in [\dd]$, by the symmetry of $\partial B_{\epsilon}(0)$ under $y_j\mapsto -y_j$,
\[
\avgint_{\partial B_{1}(0)}y_iy_j \diff y=\delta_{ij}\avgint_{\partial B_{1}(0)}y_i^2 \diff y=\frac{1}{\dd}. 
\]
It follows that
\begin{align*}
\lim_{\epsilon\to 0}\frac{\Delta_{\epsilon}f(x)}{\epsilon^2}&=\frac{1}{2}\sum_{i,j=1}^{\dd}\avgint_{\partial B_{1}(0)}\partial_{ij}^2f(x)y_iy_j \diff y=\frac{1}{2}\sum_{i=1}^{\dd}\partial_{ii}^2f(x)\avgint_{\partial B_{1}(0)}y_i^2\diff y\\
&=\frac{1}{2\dd}\sum_{i=1}^{\dd}\partial_{ii}^2f(x)=\frac{\Delta f(x)}{2\dd}.
\end{align*}

Next we move to the proof of \eqref{eq:Delta_eps_bound_appendix}. We have
\begin{align*}
\Delta_{\epsilon}f(x)&=\int_0^{\epsilon} \frac{\diff}{\diff r}\left\{\avgint_{\partial  B_{r}(0)}[f(x+y)-f(x)]\diff y\right\}\diff r\\
&=\int_0^{\epsilon} \frac{\diff}{\diff r}\left[\avgint_{\partial  B_{r}(0)}f(x+y)\diff y\right]\diff r.
\end{align*}
Since, by the change of variables $y\mapsto r y$,
\begin{align*}
\avgint_{\partial  B_{r}(0)}f(x+y)\diff y=\avgint_{\partial  B_1(0)}f(x+ry)\diff y,
\end{align*}
we have, by integration by parts,
\begin{align*}
&\frac{\diff}{\diff r}\left[\avgint_{\partial  B_{r}(0)}f(x+y)\diff y\right]=\avgint_{\partial  B_1(0)}\frac{\diff}{\diff r}f(x+ry)\diff y=\avgint_{\partial  B_1(0)}\langle\nabla f(x+ry),y\rangle\diff y\\
&=\avgint_{\partial  B_1(0)}\frac{1}{r}\langle\nabla_y [f(x+ry)],y\rangle\diff y=\avgint_{ B_1(0)}\frac{1}{r}\Delta_y[ f(x+ry)]\diff y\\
&=\frac{r}{\dd}\avgint_{B_1(0)}\Delta f(x+ry)\diff y\le \frac{\ell r}{\dd},
\end{align*}
where the last inequality used $\Delta f\le \ell$. It follows that
\[
\Delta_{\epsilon}f(x)\le \int_0^{\epsilon}r\diff r=\frac{\ell}{\dd}\frac{\epsilon^2}{2}.
\]
\end{proof}

\subsection{Smoothing under Ornstein-Uhlenbeck and heat semigroups}
\label{subsec:appendix_smooth}
In this section we discuss Proposition \ref{prop:OU_smooth}. Since most of the results in the proposition  are standard, when relevant we  will only sketch the arguments and give references for detailed proofs. For the sake of completeness we will prove some additional smoothing properties beyond those stated in Proposition \ref{prop:OU_smooth}.

\begin{proposition}
\label{prop:OU_smooth_appendix}
Let $\Gaussian$ be the standard Gaussian measure on $\R^{\dd}$, and let $f:\R^{\dd}\to \R_{\ge 0}$ be a nonnegative function in $L^1(\Gaussian)$. Let $(\Pheat_t)_{t\ge 0}$ be the Ornstein-Uhlenbeck semigroup,
\begin{equation}
\label{eq:OU_def_prop_appendix}
\Pheat_t\density(x):=\int_{\R^{\dd}} \density(e^{-t}x+\sqrt{1-e^{-2t}}y)\diff \Gaussian(y),\quad  t\ge 0,\quad x\in \R^{\dd},
\end{equation}
and let 
 $(\Hheat_t)_{t\ge 0}$ be the heat semigroup
\begin{equation}
\label{eq:heat_def_prop_appendix}
\Hheat_t\density(x):=\int_{\R^{\dd}} \density(x+\sqrt{t}y)\diff \Gaussian(y),\quad  t\ge 0,\quad x\in \R^{\dd}.
\end{equation}
\begin{enumerate}
\item  For any $x\in \R^{\dd}$ and $t>0$,
\begin{equation}
\label{eq:less_lc_appendix}
\nabla^2\log\Pheat_t\density(x) \succeq -\frac{e^{-2t}}{1-e^{-2t}} \Id_{\dd}\quad\text{and}\quad \nabla^2\log\Hheat_t\density(x) \succeq -\frac{1}{t}\Id_{\dd}. 
\end{equation}
\item Suppose that $\density$ is $\conv$-log-concave for some $\conv\in \R$ (i.e., $x\mapsto \log \density(x)-\conv\frac{|x|^2}{2}$ is concave). Then, for every $x\in \R^{\dd}$,
\begin{equation}
\label{eq:beta_log-concave_OU_appendix}
\nabla^2\log\Pheat_t\density(x)\preceq \frac{ e^{-2t}\conv}{1-\conv(1-e^{-2t})} \Id_{\dd}\begin{cases}
\text{for any }t\in [0,\infty) &\text{if }\conv\le 1\\
\text{for any }t\in \left[0,\log\left(\sqrt{\frac{\conv}{\conv-1}}\right)\right] &\text{if }\conv >1
\end{cases},
\end{equation}
and
\begin{equation}
\label{eq:beta_log-concave_heat_appendix}
\nabla^2\log\Hheat_t\density(x)\preceq \frac{\conv}{1-\conv t}\Id_{\dd}\begin{cases}
\text{for any }t\in [0,\infty) &\text{if }\conv\le 0\\
\text{for any }t\in \left[0,\frac{1}{\conv}\right] &\text{if }\conv >0
\end{cases}.
\end{equation}

\item Suppose that $\density$ is $\conv$-log-convex for some $\conv\le 0$ (i.e., $x\mapsto \log \density(x)-\conv\frac{|x|^2}{2}$ is convex).  Then,  for every $x\in \R^{\dd}$ and $t>0$,
\begin{equation}
\label{eq:beta_cvx_appendix}
\nabla^2\log\Pheat_t\density(x)\succeq \frac{e^{-2t}\conv}{1-\conv(1-e^{-2t})} \Id_{\dd}\quad\text{and}\quad \nabla^2\log\Hheat_t\density(x) \succeq \frac{\conv}{1-\conv t}\Id_{\dd}.
\end{equation}
Suppose $\density$ is $\conv\dd$-log-subharmonic for some $\conv\le 0$ (i.e.,  $x\mapsto \log f(x)-\conv\frac{|x|^2}{2}$ is subharmonic). Then, for every $x\in \R^{\dd}$ and $t>0$,
\begin{equation}
\label{eq:beta_subharmonic_appendix}
\Delta\log\Pheat_t\density(x)\ge \frac{e^{-2t}\conv\dd}{1-\conv(1-e^{-2t})}\quad\text{and}\quad \Delta \log\Hheat_t\density(x) \ge  \frac{\conv\dd}{1-\conv t}.
\end{equation}
\end{enumerate}
\end{proposition}
\begin{remark}
It was shown in \cite[Remark]{MR4632741} that there exists a $0$-log-superharmonic $\density$ such that, for any $t>0$, $\Pheat_t\density$ is strictly log-subharmonic. Concretely, let $\dd=2$ and $\density(x)=\density(x_1,x_2):=e^{x_1x_2}$.  Then, $\Delta\log \density(x)=0$ for all $x\in \R^{\dd}$ while $\Delta\log \Pheat_t\density(x)>0$ for all $x\in\R^{\dd}$ and $t>0$. We conclude that log-superharmonic functions are not preserved under the Ornstein-Uhlenbeck/heat semigroups.
\end{remark}

\begin{proof}
We will present a number of proofs for the various parts in Proposition \ref{prop:OU_smooth_appendix}. Since the Ornstein-Uhlenbeck semigroup and heat semigroup are related by the change of variables
\begin{equation}
\label{eq:OU_heat}
\Pheat_t\density(x)=\Hheat_{1-e^{-2t}}\density (e^{-t}x),
\end{equation}
we will use  whichever one is convenient for the specific proof. The most robust proof technique (which works also on manifolds) is due to Hamilton \cite{hamilton1993matrix, MR4370325} based on deriving a partial differential equation for $M(t,x):=\nabla^2\log\Hheat_t\density(x)$, and then using the maximum principle. Standard computation (e.g. \cite[Proposition 2.11]{muller2006differential}) shows that
\begin{equation}
\label{eq:PDE_Hessian}
\partial_tM(t,x)=\frac{1}{2}\Delta M(t,x)+L(M(t,x))+M^2(t,x),
\end{equation} 
where $L(M(t,x))$ is a \emph{linear} differential operator in $M(t,x)$. Below we will apply the maximum principle to $(t,x)\mapsto M(t,x)$ to deduce Proposition \ref{prop:OU_smooth_appendix}(1-3). 

Another proof technique, which is more probabilistic  in nature, uses the specific form of the Ornstein-Uhlenbeck/heat semigroup in Euclidean space. Specifically, it is based on the following covariance identity 
\begin{equation}
\label{eq:cov_identity}
\nabla^2\log\Pheat_t\density(x)=\frac{e^{-2t}}{1-e^{-2t}}\left(\frac{\Cov\left[p_{e^{-t}x,1-e^{-2t}}\right]}{1-e^{-2t}}-\Id_{\dd}\right), \quad \forall~x\in\R^{\dd},~\forall~ t>0,
\end{equation}
where $\Cov\left[p_{z,s}\right]$ is the covariance matrix of the probability measure
\begin{equation}
\label{eq:conditional_measure}
\diff p_{z,s}(y)\propto \density(y)e^{-\frac{|y-z|^2}{2s}}\diff y.
\end{equation}
The identity \eqref{eq:cov_identity}  is standard, e.g.,  \cite[Equation (3.2)]{MR4797372}, \cite[Equation (3.4)]{MR4748461}.\\

\noindent\textbf{Proofs of Proposition \ref{prop:OU_smooth_appendix}(1).} 

\noindent\textbf{Proof 1.} This is Hamilton's matrix inequality  \cite{MR4370325} which is proven via the maximum principle. Indeed, by \eqref{eq:PDE_Hessian}, since  the smallest eigenvalue \(\lambda\) is a concave function of   \(M(x,t)\), it satisfies
\begin{equation}
  \label{eq:min_eigenvalue}
\partial_{t}\lambda(x,t)\ge \frac{1}{2}\Delta \lambda (x,t) +L(\lambda (x,t))+ \lambda^{2}(x,t).
\end{equation}
Hence, $\lambda(x,t)\ge g(t)$  where \(g(t):=-1/t\) solves the equation,
\[
\partial_tg(t)=g^{2}(t), \qquad g(0)=-\infty.
\]
\noindent\textbf{Proof 2.} Use \eqref{eq:cov_identity} and $\Cov\left[p_{e^{-t}x,e^{-2t}}\right]\succeq 0$.

\noindent\textbf{Proof 3.} Consequence of the intrinsic dimensional local logarithmic Sobolev inequality: The term inside the logarithm in \cite[Equation 29]{MR4698563} must be nonnegative. (This is analogous to the way in which the Li-Yau inequality is deduced from the dimensional local logarithmic Sobolev inequality \cite{MR3612336}.)

\noindent\textbf{Proof 4.} Mixture of log-convex densities is log-convex: For any fixed $y\in \R^{\dd}$ the function $x\mapsto \frac{e^{-2t}}{1-e^{-2t}}\frac{|x|^2}{2}+\log \Pheat_t\delta_y(x)$ is convex, where $\delta_y$ is a point mass at $y$. Hence, the mixture $x\mapsto\frac{e^{-2t}}{1-e^{-2t}}\frac{|x|^2}{2}+\log \int\density(y)\Pheat_t\delta_y(x)\diff y$ is also convex for any probability density $\density:\R^{\dd}\to \R_{\ge 0}$. The proof is complete since $ \int\density(y)\Pheat_t\delta_y(x)\diff y=\Pheat_t\density(x)$ \cite[Lemma 1.3, Appendix]{MR3782065}. \\

\noindent\textbf{Proofs of Proposition \ref{prop:OU_smooth_appendix}(2).} 

\noindent\textbf{Proof 1.}  We will use the maximum principle. Let $\Lambda(t,x)$ be the maximal eigenvalue of $M(t,x):=\nabla^2\log\Hheat_t\density(x)$. Since the maximal eigenvalue $\Lambda$ is a convex function of  $M$, \eqref{eq:PDE_Hessian} implies that
\begin{equation}
\label{eq:eigen_min}
\partial_t\Lambda(t,x)\le \frac{1}{2}\Delta \Lambda(t,x)+L(\Lambda(t,x))+\Lambda^2(t,x).
\end{equation}
Fix $x\in \R^{\dd}$. Then,
\begin{equation*}
\Lambda(t,x)\le g(t)\quad\text{where $g$ solves the equation} \quad \partial_tg(t)=g^2(t),\quad g(0)=\Lambda(0,x),
\end{equation*}
for every $t$ for which $g$ is well-defined. Since the solution of the ordinary differential equation for $g$ is $g(t)=\frac{\Lambda(0,x)}{1-\Lambda(0,x)t}$ for all $t$ where the denominator does not vanish, we see that if $\Lambda(0,x)\le 0$, then $g$ is well-defined for all $t\ge 0$, and if $\Lambda(0,x)> 0$, then $g$ is well-defined for all $t\in \left[0,\frac{1}{\Lambda(0,x)}\right)$. Hence, since $\Lambda(0,x)\le \conv$,
\begin{equation*}
\Lambda(t,x)\le\frac{\Lambda(0,x)}{1-\Lambda(0,x)t}
\begin{cases}
\text{for any }t\in [0,\infty) &\text{if }\conv\le 0\\
\text{for any }t\in \left[0,\frac{1}{\conv}\right] &\text{if }\conv >0,
\end{cases}
\end{equation*}
which completes the proof of \eqref{eq:beta_log-concave_heat_appendix}. 

\noindent\textbf{Proof 2.} We will use the covariance identity  following the argument in \cite[Lemma 3.4(2)]{MR4797372}. If $\density$ is $\conv$-log-concave then the measure $p_{z,s}$ in \eqref{eq:conditional_measure} is $\left(\frac{1}{s}-\conv\right)$-log-concave. In particular, $p_{e^{-t}x,1-e^{-2t}}$ is $\left(\frac{1}{1-e^{-2t}}-\conv\right)$-log-concave so by the Brascamp-Lieb inequality \cite{MR450480}, as long as $\left(\frac{1}{1-e^{-2t}}-\conv\right)\ge 0$,
\[
\Cov\left[p_{e^{-t}x,1-e^{-2t}}\right] \preceq \left(\frac{1}{1-e^{-2t}}-\conv\right)^{-1}\Id_{\dd}.
\]
Hence, by \eqref{eq:cov_identity}, 
\begin{equation*}
\nabla^2\log\Pheat_t\density(x)\preceq \frac{\conv e^{-2t}}{1-\conv(1-e^{-2t})} \Id_{\dd}
\begin{cases}
\text{for any }t\in [0,\infty) &\text{if }\conv\le 1\\
\text{for any }t\in \left[0,\log\left(\sqrt{\frac{\conv}{\conv-1}}\right)\right] &\text{if }\conv >1,
\end{cases}
\end{equation*}
which proves \eqref{eq:beta_log-concave_OU_appendix}.

\noindent\textbf{Proof 3.} We will use the Pr\'ekopa-Leindler inequality. We need to show that $x\mapsto\Hheat_t\density(x)e^{-c_t\frac{|x|^2}{2}}$ is log-concave were $c_t:=\frac{\conv}{1-\conv t}$. We have
\begin{align*}
&\Hheat_t\density(x)e^{-\frac{\conv}{1-\conv t}\frac{|x|^2}{2}}=\int \density(z)e^{-\frac{|z-x|^2}{2t}}e^{-\frac{\conv}{1-\conv t}\frac{|x|^2}{2}}\diff z\\
&=\int \density(z)\exp\left(-\frac{1}{2}\left[\frac{1}{t(1+tc_t)}\left|z-(1+tc_t)x\right|^2+\frac{c_t}{1+tc_t}|z|^2\right]\right)\diff z=\Hheat_{t(1+tc_t)}\density_{c_t}((1+tc_t)x),
\end{align*} 
where $\density_{c_t}(z)=\density(z)e^{-\frac{1}{2}\frac{c_t}{1+tc_t}|z|^2}$, as long as $1+tc_t\ge 0$. Since $\density$ is $\conv$-log-concave, the function $\density_{c_t}$  is  $0$-log-concave as $\conv-\frac{c_t}{1+tc_t}=\conv-\frac{\frac{\conv}{1-\conv t}}{1+t\frac{\conv}{1-\conv t}}=0$. Hence, $\Hheat_{t(1+tc_t)}\density_{c_t}$ is log-concave because it is the convolution of log-concave functions,  which is log-concave by the Pr\'ekopa-Leindler inequality \cite[\S 9]{gardner2002brunn}. Hence, as long as $1+tc_t\ge 0$, $\Hheat_t\density(x)e^{-\frac{\conv}{1-\conv t}\frac{|x|^2}{2}}$ is log-concave, which implies \eqref{eq:beta_log-concave_heat_appendix}. \\

\noindent\textbf{Proofs of Proposition \ref{prop:OU_smooth_appendix}(3).} 

\noindent\textbf{Proof 1.}  We will use the maximum principle. Recall that the minimum eigenvalue  $\lambda(t,x)$ of $M(t,x):=\nabla^2\log\Hheat_t\density(x)$ satisfies \eqref{eq:min_eigenvalue}. Hence,   $\lambda(t,x)\ge g(t)$ where $g$ solves the equation $\partial_tg(t)=g^2(t)$ with $g(0)=\lambda(0,x)$.  Since the solution of the ordinary differential equation for $g$ is $g(t)=\frac{\lambda(0,x)}{1-\lambda(0,x)t}$ for all $t$ where the denominator  does not vanish, we get that if $\lambda(0,x)\le 0$, then, for all $t\ge 0$, $\lambda(t,x)\ge\frac{\lambda(0,x)}{1-\lambda(0,x)t}$. Since $\lambda(0,x)\ge \conv$ the proof of \eqref{eq:beta_cvx_appendix} is complete.

To prove \eqref{eq:beta_subharmonic_appendix} let $m(t,x):=\Tr[M(t,x)]$, and note that, by $\Tr[M^2(t,x)]\ge \frac{\Tr[M(t,x)]^2}{\dd}$, \eqref{eq:PDE_Hessian} implies that
\begin{equation}
\label{eq:eigen_min}
\partial_tm(t,x)\ge \frac{1}{2}\Delta m(t,x)+L(m(t,x))+\frac{m^2(t,x)}{\dd}.
\end{equation}
Again, this implies that, for any fixed $x\in \R^{\dd}$ and $t\ge 0$,  $m(t,x)\ge \frac{m(0,x)}{1-\frac{m(0,x)}{\dd}t}$. Since $m(t,x)\ge\conv\dd$ the proof of  \eqref{eq:beta_subharmonic_appendix} is complete

\noindent\textbf{Proof 2.} The proof is based on the fact that a mixture of log-convex (res. log-subharmonic) functions is log-convex (res. log-subharmonic): The argument we present below is based on the proof of \cite[Lemma 5]{MS2022}  which treats the $\conv$-log-convex case. Here we adapt the proof also to the $\conv\dd$-log-subharmonic case. Our starting point is the following identity for the action of Ornstein-Uhlenbeck semigroup. The point of this identity to separate the effect of the semigroup on the quadratic part of the function. 
\begin{lemma}
\label{lem:Ptf_identity}
Fix $\beta\ge 0$ and let $\density(x)=e^{-R(x)-\beta\frac{|x|^2}{2}}$. Let $(\Pheat_t)$ be the Ornstein-Uhlenbeck semigroup \eqref{eq:OU_def_prop_appendix}. Then, 
\begin{align}
\label{eq:P_t_rep_quadratic}
\Pheat_t\density(x)=e^{-\frac{\beta e^{-2t}}{1+(1-e^{-2t})\beta}\frac{|x|^2}{2}}\int e^{-R_{t,z}(x)}e^{-\frac{\beta}{2}|z|^2}\diff z,
\end{align}
where
\begin{align}
\label{eq:R_tz_def}
R_{t,z}(x):=(2\pi)^{-\frac{\dd}{2}}(1-e^{-2t}+\beta^{-1})^{-\frac{\dd}{2}}R\left(\frac{\sqrt{1-e^{-2t}}}{\sqrt{1-e^{-2t}+\beta^{-1}}}z+\frac{e^{-t}}{1+\beta(1-e^{-2t})}x\right).
\end{align}
\end{lemma}
\begin{proof}
 By definition,
\begin{align*}
\Pheat_t\density(x)&=\int \density(e^{-t}x+\sqrt{1-e^{-2t}}y)\diff \Gaussian(y)\\
&=\frac{1}{(2\pi)^{\frac{\dd}{2}}}\int e^{-R(e^{-t}x+\sqrt{1-e^{-2t}}y)-\beta\frac{|e^{-t}x+\sqrt{1-e^{-2t}}y|^2}{2}-\frac{|y|^2}{2}}\diff y\\
&=\frac{e^{-\beta e^{-2t}\frac{|x|^2}{2}}}{(2\pi)^{\frac{\dd}{2}}}\int e^{-R(e^{-t}x+\sqrt{1-e^{-2t}}y)-\beta\frac{2\langle e^{-t}x,\sqrt{1-e^{-2t}}y\rangle +(1-e^{-2t}+\beta^{-1})|y|^2}{2}}\diff y\\
&=\frac{e^{-\beta e^{-2t}\frac{|x|^2}{2}}}{(2\pi)^{\frac{\dd}{2}}}\int e^{-R(e^{-t}x+\sqrt{1-e^{-2t}}y)-\frac{\beta}{2}\left|\sqrt{1-e^{-2t}+\beta^{-1}}y+\frac{\sqrt{1-e^{-2t}}e^{-t}}{\sqrt{1-e^{-2t}+\beta^{-1}}}x\right|^2}e^{-\frac{-(1-e^{-2t})e^{-2t}}{1-e^{-2t}+\beta^{-1}}\frac{\beta|x|^2}{2}}\diff y\\
&=\frac{e^{-\left(1-\frac{1-e^{-2t}}{1-e^{-2t}+\beta^{-1}}\right)\frac{e^{-2t}\beta}{2}|x|^2}}{(2\pi)^{\frac{\dd}{2}}}\int e^{-R(e^{-t}x+\sqrt{1-e^{-2t}}y)-\frac{\beta}{2}\left|\sqrt{1-e^{-2t}+\beta^{-1}}y+\frac{\sqrt{1-e^{-2t}}e^{-t}}{\sqrt{1-e^{-2t}+\beta^{-1}}}x\right|^2}\diff y\\
&=\frac{e^{-\frac{\beta e^{-2t}}{1+(1-e^{-2t})\beta}\frac{|x|^2}{2}}}{(2\pi)^{\frac{\dd}{2}}}\int e^{-R(e^{-t}x+\sqrt{1-e^{-2t}}y)-\frac{\beta}{2}\left|\sqrt{1-e^{-2t}+\conv^{-1}}y+\frac{\sqrt{1-e^{-2t}}e^{-t}}{\sqrt{1-e^{-2t}+\beta^{-1}}}x\right|^2}\diff y.
\end{align*}

Let $z=\sqrt{1-e^{-2t}+\beta^{-1}}y+\frac{\sqrt{1-e^{-2t}}e^{-t}}{\sqrt{1-e^{-2t}+\beta^{-1}}}x$, so that $\diff y=(1-e^{-2t}+\beta^{-1})^{-\frac{\dd}{2}}\diff z$, and change variables to get
\begin{align*}
\Pheat_t\density(x)&=\frac{e^{-\frac{\beta e^{-2t}}{1+(1-e^{-2t})\beta}\frac{|x|^2}{2}}}{(2\pi)^{\frac{\dd}{2}}}\int e^{-R\left(\frac{\sqrt{1-e^{-2t}}}{\sqrt{1-e^{-2t}+\beta^{-1}}}z+\frac{e^{-t}}{1+\beta(1-e^{-2t})}x\right)-\frac{\beta}{2}|z|^2}(1-e^{-2t}+\beta^{-1})^{-\frac{\dd}{2}}\diff z\\
&=e^{-\frac{\beta e^{-2t}}{1+(1-e^{-2t})\beta}\frac{|x|^2}{2}}\int e^{-R_{t,z}(x)}e^{-\frac{\beta}{2}|z|^2}\diff z,
\end{align*}
where
\[
R_{t,z}(x):=(2\pi)^{-\frac{\dd}{2}}(1-e^{-2t}+\beta^{-1})^{-\frac{\dd}{2}}R\left(\frac{\sqrt{1-e^{-2t}}}{\sqrt{1-e^{-2t}+\beta^{-1}}}z+\frac{e^{-t}}{1+\beta(1-e^{-2t})}x\right).
\]
\end{proof}
Let us now complete the proof of Proposition \ref{prop:OU_smooth_appendix}(3) using \eqref{eq:P_t_rep_quadratic}. Given $\density$ which is $\conv$-log-convex (res. $\conv\dd$-log-subharmonic) let $R(x):=-\log\density (x)-|\conv|\frac{|x|^2}{2}$. Then $R$ is concave (res. superharmonic), so with $R_{t,z}$ as in \eqref{eq:R_tz_def} we have that $x\mapsto e^{-R_{t,z}(x)}$ is log-convex (res. log-subharmonic). Since the mixture of log-convex functions is log convex \cite[Chapter 16.B]{marshall2011inequalities} (res. the mixture of log-subharmonic functions is log-subharmonic \cite[Proposition 2.2]{MR2578455}), it follows that $x\mapsto \int e^{-R_{t,z}(x)}e^{-\frac{|\conv|}{2}|z|^2}\diff z$ is log-convex (res. log-subharmonic). From \eqref{eq:P_t_rep_quadratic} we see that $\Pheat_t\density$ is $\frac{\conv e^{-2t}}{1-\conv(1-e^{-2t})}$-log-convex (res. $\dd\frac{\conv e^{-2t}}{1-\conv(1-e^{-2t})}$-log-subharmonic).
\end{proof}

\bibliographystyle{amsplain0}
\bibliography{ref_majorization_transport}

\end{document}